\documentclass[a4paper,reqno]{amsart}

\usepackage{tikz-cd}
\usepackage{amsmath}
\usepackage{amssymb}
\usepackage{enumerate}
\usepackage{ascmac}
\usepackage{mathrsfs}
\usepackage{multirow}
\usepackage{amsthm}
\usepackage{wrapfig}
\usepackage{here}
\usepackage{multirow}
\usepackage{hyperref}
\usepackage{graphicx,color}
\usepackage{caption}

\theoremstyle{plain}
    \newtheorem{thm}{Theorem}[section]
    \newtheorem{lem}[thm]{Lemma}
    \newtheorem{prop}[thm]{Proposition}
    \newtheorem{cor}[thm]{Corollary}

\theoremstyle{definition}    
    \newtheorem{defn}[thm]{Definition}

\def\an{{\mathrm{an}}}
\def\aru{{\arrow[u,"\rotatebox{90}{$\in$}",phantom]}}
\def\Aut{{\mathrm{Aut}}}
\def\B{{\mathcal{B}}}
\def\C{{\mathbb{C}}}
\def\Cc{{\mathcal{C}}}

\def\CH{{\mathrm{CH}}}

\def\Dc{{\mathscr{D}}}
\def\dec{{\mathrm{dec}}}
\def\del{\partial}

\def\div{{\mathrm{div}}}
\def\E{{\mathcal{E}}}
\def\ev{{\mathrm{ev}}}

\def\J{{\mathcal{J}}}
\def\id{{\mathrm{id}}}
\def\ind{{\mathrm{ind}}}

\def\L{{\mathcal{L}}}
\def\lra{\:{\longrightarrow}\:}
\def\M{{\mathcal{M}}}
\def\NS{{\mathrm{NS}}}
\def\O{{\mathcal{O}}}
\def\P{{\mathbb{P}}}
\def\Pc{{\mathcal{P}}}
\def\Q{{\mathbb{Q}}}
\def\Qc{{\mathcal{Q}}}
\def\R{{\mathbb{R}}}
\def\rank{{\mathrm{rank}\:}}
\def\Spec{{\mathrm{Spec}\,}}
\def\tr{{\mathrm{tr}}}
\def\X{{\mathcal{X}}}
\def\Y{{\mathcal{Y}}}
\def\Z{{\mathbb{Z}}}
\def\Zc{{\mathcal{Z}}}

\begin{document}
\title{Higher Chow cycles on Eisenstein $K3$ surfaces}
\author{Ken Sato}
\address{Department of Mathematics, Institute of Science Tokyo}
\email{sato.k.da@m.titech.ac.jp}
\begin{abstract}
We construct higher Chow cycles of type $(2,1)$ on some families of $K3$ surfaces with non-symplectic automorphisms of order 3 and prove that our cycles are indecomposable for very general members.
The proof is a combination of some degeneration arguments, and explicit computations of the regulator map.
\end{abstract}
\maketitle
\section{Introduction}
Higher Chow groups $\CH^p(X,q)$ of smooth varieties $X$ introduced by Bloch are a generalization of the classical algebraic cycles.
They are related to many important invariants in algebraic geometry, $K$-theory, and number theory.
However, when the codimension $p$ is greater than 1, its structure is still mysterious.
In the papers \cite{MS23} and \cite{Sato}, we explicitly construct higher Chow cycles of type $(2,1)$ on $K3$ surfaces with non-symplectic automorphisms of order $2$ and $4$, respectively.
In these papers, it turned out that non-symplectic automorphisms on $K3$ surfaces can play an effective role for explicit constructions of $(2,1)$-cycles.

In this paper, we consider the case of order 3.
A pair $(X,\sigma)$ of a $K3$ surface $X$ and a non-symplectic automorphism $\sigma$ of order $3$ is called an \textit{Eisenstein K3 surface}. 
By the classification theory (\cite{AS08}, \cite{Tak11} and \cite{DK07}), each component of the moduli space is labelled by a pair $(r,a)$ of natural numbers, and the corresponding component $\M_{r,a}$ is a Zariski open subset of a quotient of a complex ball by a unitary group.
The invariant $r$ is even number, and $(r,a)$ satisfies the inequalities $a\le (r+2)/2$ and $a \le (22-r)/2$ except for one case.
(See Figure \ref{Eisdist}.)

Here we focus on Eisenstein $K3$ surfaces on the extremal lines $a = (r+2)/2$ for $r\le 10$ and $a = (22-r)/2$ for $r\ge 12$ in this paper.
For each such $(r,a)$, we construct an explicit $(2,1)$-cycle $\xi$ on the members $(X,\sigma)$ of $\M_{r,a}$.
Then we prove the following.
\begin{thm}\label{mainintro2}
For a very general $(X,\sigma)$, the indecomposable parts of the $(2,1)$-cycles $\xi$ and $\sigma_*(\xi)$ are non-torsion and linearly independent.
In particular, the rank of the indecomposable parts of $\CH^2(X,1)$ is equal or greater than $2$.
\end{thm}

Here, the \textit{indecomposable part} of $\CH^2(X,1)$ is defined as the quotient of $\CH^2(X,1)$ by $\mathrm{Pic}(X)\otimes \C^\times$, and the indecomposable part of a $(2,1)$-cycle is its image in the indecomposable part.

For the construction of  cycles, we use the model of the Eisenstein $K3$ surfaces as triple covers of $\P^1\times \P^1$  branching along nodal curves $B$ of bidegree $(3,3)$.
The first example of such construction is given by Kond\=o (\cite{Kon}), and generalized in \cite{MOT}.
Let $F_1$ (resp. $F_2$) be a divisor of type $(1,0)$ (resp. $(0,1)$) on $\P^1\times \P^1$ which intersect with $B$ at a point with multiplicity $2$.
Then the pull-back $C_i$ of $F_i$ is a rational curve on the Eisenstein $K3$ surfaces, and $C_1$ and $C_2$ intersect at 3 points.
We can construct a $(2,1)$-cycle from such configuration of rational curves.
In general, cyclic cover model provides a convenient source for constructing $(2,1)$-cycles.
See \cite{Sre14}, \cite{Sat24}, \cite{MS23},  \cite{Sre24}, \cite{Sato} and \cite{Sree} for similar methods in other cases.

The proof of indecomposability  is divided into two steps.
The first part is a degeneration argument which reduces the general case of Theorem \ref{mainintro2} to the  starting case $r=18$.
The second part is an explicit calculation of the \textit{Abel-Jacobi map} in the case $r=18$ to show that the $(2,1)$-cycle is indecomposable for this case.

A key device of the first part is the Eisenstein Jacobian $J(X,\sigma)$, which is a generalized complex torus associated with the Eisenstein lattice of $(X,\sigma)$, together with corresponding Abel-Jacobi map
\begin{equation}
\begin{tikzcd}
\nu_E:\CH^2(X,1)\arrow[r] & J(X,\sigma).
\end{tikzcd}
\end{equation}
The advantages of using Eisenstein Jacobian are that it behaves continuously in a family of Eisenstein $K3$ surfaces and it can be used to detect indecomposability for very general cases.
We prove a degeneration lemma (Lemma \ref{degenerationlemma}), which enables us to reduce the Theorem \ref{mainintro2} to some inductive arguments on $r$ as in \cite{MS23}.
The technical novelty of this paper is that this degeneration lemma can be applied to the wider class of degenerations than those of \cite{MS23}.

The indecomposablity of starting case ($r=18$) is proved by the direct calculation of the regulator map.
In this case, the image of our $(2,1)$-cycles by the regulator map is expressed by the ``incomplete integral" considered in \cite{MT}.
This integral expression enables us to prove that the normal function $G(\lambda)$ associated with our cycles satisfies the following \textit{inhomogeneous hypergeometric differential equation}:
\begin{equation}\label{PFintro}
\lambda(1-\lambda)\dfrac{d^2G}{d\lambda^2} + \left(1-\frac{7}{3}\lambda\right)\dfrac{dG}{d\lambda} -\frac{4}{9}G(\lambda) = \frac{1-\zeta}{\lambda-1}.
\end{equation}
If the right-hand side of (\ref{PFintro}) is zero, this is the Gauss hypergeometric differential equation of type $(2/3,2/3,1)$.
Since the periods of the Eisenstein $K3$ surfaces of the case $r=18$ satisfies the Gauss hypergeometric differential equation of type $(2/3,2/3,1)$, the normal function $G(\lambda)$ and period functions are linearly independent.
This implies that our cycles are indecomposable for very general cases when $r=18$, and we complete the proofs.
The function $G(\lambda)$ is an example of special functions related to hypergeometric functions appearing as normal functions of families of $(2,1)$-cycles.

The structure of this paper is as follows.
In Section 2, we recall basic results on Eisenstein $K3$ surfaces and higher Chow cycles.
In Section 3, we define the Eisenstein Jacobian $J(X,\sigma)$ and prove that we can detect indecomposable cycles via the Eisenstein Jacobian in very general cases.
In Section 4, we introduce the cyclic covering construction of Eisenstein $K3$ surfaces and explain the construction of our $(2,1)$-cycles.
In Section 5, we define a dominant family $\X\rightarrow \hat{U}_r$ of Eisenstein $K3$ surfaces of type $(r,a)$.
In Section 6, we prove the degeneration lemma, which is the crucial for our induction argument.
Then we state the main theorem, and run the induction part of the proof.
In section 7, we calculate the image of our $(2,1)$-cycles under the regulator map for the starting case of the induction, and complete the proof of the main theorem.

\subsection*{Acknowledgement}
The author is very grateful to Shohei Ma for many valuable suggestions, discussions and encouragement.
This work was supported by JSPS KAKENHI 21H00971.

\subsection*{Conventions}
Throughout this paper, we fix a primitive 3rd root of unity $\zeta = \exp\left(2\pi \sqrt{-1}/3\right)$.
``A very general point" on analytic varieties means points on the complement of a union of countably many analytic subsets.

\section{Preliminaries}
In this section, we recall some properties of Eisenstein $K3$ surfaces and higher Chow cycles.
\subsection{Eisenstein $K3$ surfaces}
Let $X$ be a complex $K3$ surface and $\sigma\in \Aut(X)$ be an automorphism of order 3 such that $\sigma^*\omega = \zeta\omega$ where $\omega$ is a non-zero holomorphic 2-form on $X$.
We call such a pair $(X,\sigma)$ an \textit{Eisenstein K3 surface}.
For an Eisenstein $K3$ surface $(X,\sigma)$, let $H^2(X,\C) = H_{1}\oplus H_{\zeta}\oplus H_{\zeta^2}$ be the eigenspace decomposition with respect to $\sigma^*$.
We define the sublattices $L(X,\sigma)$ and $E(X,\sigma)$ of $H^2(X,\Z)$ as follows:
\begin{equation}
\begin{aligned}
&L(X,\sigma) = H_{1}\cap H^2(X,\Z) = H^2(X,\Z)^{\sigma^*} \\
&E(X,\sigma) = (H_{\zeta}\oplus H_{\zeta^2})\cap H^2(X,\Z).\\
\end{aligned}
\end{equation}
The lattice $E(X,\sigma)$ is the orthogonal complement of $L(X,\sigma)$ in $H^2(X,\Z)$.
By the $\sigma^*$-action on $H^2(X,\Z)$, the Eisenstein integer ring $\Z[\zeta]$ naturally acts on $E(X,\sigma)$.
We have the natural Hermitian form $h$ on $E(X,\sigma)$ defined by $h(x,y)=\langle x,y\rangle+\zeta\langle x,\zeta y\rangle+\zeta^2\langle x,\zeta^2y\rangle \in \Z[\zeta]$ for $x,y\in E(X,\sigma)$.
Since $[\omega]\in H^2(X,\C)$ is contained in $H_{\zeta} = E(X,\sigma)\otimes_{\Z[\zeta]} \C$, we have 
\begin{equation}
L(X,\sigma)\subset \NS(X).
\end{equation}

For an Eisenstein $K3$ surface $(X,\sigma)$, the invariants $r,a$ are defined so that the following equations holds.
\begin{equation}
r = \rank L(X, \sigma),\quad E(X,\sigma)^\vee/E(X,\sigma) \simeq (\Z/3\Z)^{\oplus a}.
\end{equation}
The signature of $E(X,\sigma)$ is $(2,20-r)$.
By \cite{AS08} (Theorem 3.3) or \cite{Tak11} (Lemma 2.3), we know that the possible invariants $(r,a)$ are shown in the Figure \ref{Eisdist}.
Furthermore, for each $(r,a)$ in Figure  \ref{Eisdist}, the isomorphism class of $E(X,\sigma)$ (as Eisenstein lattices with Hermitian forms) is uniquely determined.
Hereafter $(E_{r,a},h)$ denotes the abstract Eisenstein lattice with Hermitian form which is isomorphic to $E(X,\sigma)$.
Then there exists the unique primitive embedding of $E_{r,a}$ in the $K3$ lattice $\Lambda_{K3}=U^{\oplus 3}\oplus E_8^{\oplus 2}$.
In this paper, we treat the following cases.
\begin{equation*}
a = \frac{r+2}{2} \text{ for } r\le 10\quad  \text{ or } \quad a = \frac{22-r}{2} \text{ for } 12\le r \le 18.
\end{equation*}

\begin{figure}
\includegraphics[width = 70mm]{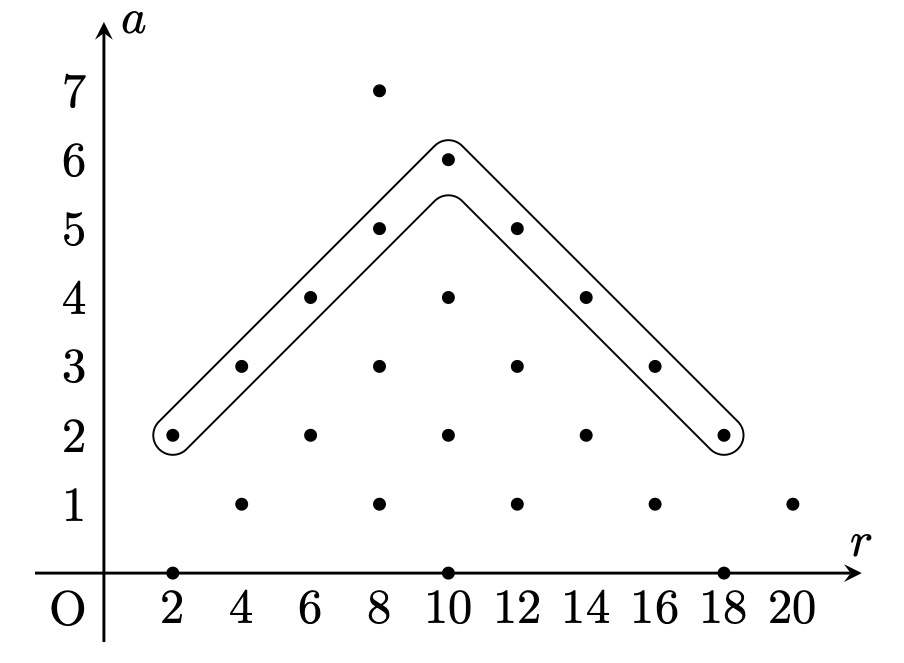}
\caption{The possible invariants $(r,a)$ and the cases we treat in this paper}
\label{Eisdist}
\end{figure}

By the Torelli theorem for $K3$ surfaces, there exists the coarse moduli space $\M_{r,a}$ of Eisenstein $K3$ surfaces of type $(r,a)$ (\cite{DK07} Theorem 11.2).
We briefly recall its construction.
Since $E_{r,a}$ has the unique primitive embedding in $\Lambda_{K3}$, for any marking $\alpha: H^2(X,\Z)\xrightarrow{\:\sim\:} \Lambda_{K3}$, $\alpha$ induces the isometry $E(X,\sigma)\xrightarrow{\:\sim\:} E_{r,a}$.
Since $[\omega] \in  E(X,\sigma)\otimes_{\Z[\zeta]}\C$, $E_{r,a}\otimes_{\Z[\zeta]}\C$ is the period domain.
Let $\B = \{[v]\in \P(E_{r,a}\otimes_{\Z[\zeta]}\C) : \langle v, \overline{v} \rangle >0\}$.
Since the signature of $E_{r,a}$ is $(2,20-r)$, $\B$ is isomorphic to a $10-r/2$-dimensional complex ball (cf. \cite{DK07}, Section 7).
Let $\Gamma$ be the unitary group $U(E_{r,a},h)$ over $\Z[\zeta]$.
The group $\Gamma$ naturally acts on the period domain $\B$ properly discontinuously.
Then the moduli space $\M_{r,a}$ is realized as a Zariski open subset of $\B/\Gamma$.
Since a very general point $[v]\in \B$ satisfies that $\langle v, x\rangle \neq 0$ for any $x\in E_{r,a}$, we have the following.
\begin{prop}\label{NSprop}
For a very general members of $\M_{r,a}$, the corresponding Eisenstein $K3$ surface $(X,\sigma)$ satisfies $\NS(X)=L(X,\sigma)$.
\end{prop}

\subsection{Higher Chow cycles}
For a smooth variety $X$ over $\C$, let $\CH^p(X,q)$ be the higher Chow group defined by Bloch (\cite{bloch}).
In this paper, we consider the case $(p,q)=(2,1)$.
An element of $\CH^2(X,1)$ is called a \textit{$(2,1)$-cycle}.
The higher Chow group $\mathrm{CH}^2(X,1)$ is isomorphic to the middle homology group of the following complex. For the proof, see, e.g., \cite{MS2} Corollary 5.3.
\begin{equation*}
K_2^{\mathrm{M}} (\C(X)) \xrightarrow{T} \displaystyle\bigoplus_{C\in X^{(1)}}\C(C)^\times \xrightarrow{\mathrm{div}}\displaystyle\bigoplus_{p\in X^{(2)}}\Z\cdot p 
\end{equation*}
Here $X^{(r)}$ denotes the set of all irreducible closed subsets of $X$ of codimension $r$. 
The map $T$ denotes the tame symbol map from the Milnor $K_2$-group of the function field $\C(X)$.
Therefore, each $(2,1)$-cycle is represented by a formal sum
\begin{equation}\label{formalsum}
\sum_j (C_j, f_j)\in \displaystyle\bigoplus_{C\in X^{(1)}}\C(C)^\times
\end{equation}
where $C_j$ are prime divisors on $X$ and $f_j\in \C(C_j)^\times$ are non-zero rational functions on them such that $\sum_j {\mathrm{div}}_{C_j}(f_j) = 0$ as codimension 2 cycles on $X$.

The intersection product on higher Chow groups induces the following map.
\begin{equation}\label{intersectionproduct}
\mathrm{Pic}(X)\otimes_\Z \Gamma(X,\mathcal{O}_X^\times)=\CH^1(X)\otimes_\Z \CH^1(X,1) \longrightarrow \mathrm{CH}^2(X,1)
\end{equation}
The image of this map is called the \textit{decomposable part} of $\mathrm{CH}^2(X,1)$ and is denoted by $\CH^2(X,1)_\dec$.
A \textit{decomposable cycle} is an element of $\CH^2(X,1)_\dec$.
The quotient $\CH^2(X,1)/\CH^2(X,1)_\dec$ is called the \textit{indecomposable part} of $\CH^2(X,1)$ and is denoted by $\mathrm{CH}^2(X,1)_{\mathrm{ind}}$. 
For a $(2,1)$-cycle $\xi$, $\xi_\ind$ denotes its image in $\mathrm{CH}^2(X,1)_{\mathrm{ind}}$. 

Let $C$ be a prime divisor on $X$, $\alpha\in \Gamma(X,\O_X^\times)$ and $[C]\in \mathrm{Pic}(X)$ be the class corresponding to $C$. 
The image of $[C]\otimes \alpha$ under (\ref{intersectionproduct}) is represented by $(C,\alpha|_C)$ in the presentation (\ref{formalsum}).

\subsection{The regulator map}
For a $\Z$-Hodge structure $H=(H_\Z,F^\bullet)$ of weight 2, the generalized intermediate Jacobian of $H$ is defined as the quotient group 
\begin{equation}
J(H) = \frac{H_\C}{H_\Z + F^2H_\C}.
\end{equation}
This is a generalized complex torus, i.e., a quotient of a $\C$-linear space by a discrete lattice.
Let $X$ be a complex projective surface such that $H^1(X,\Z)=0$ (e.g. $K3$ surface).
We write $J(X) = J(H^2(X,\Z))$ and simply call it the \textit{Jacobian} of $X$.
The following \textit{Beilinson regulator map} plays a role of the Abel-Jacobi map in the study of $(2,1)$-cycles (cf. \cite{KLM}).
\begin{equation}\label{regulatormap}
\nu: \mathrm{CH}^2(X,1) \longrightarrow J(X)
\end{equation}
We recall a formula for the regulator map following \cite{Le} pp. 458--459.
By the intersection product, the Jacobian of $X$ is isomorphic to 
\begin{equation}\label{Delignefunctional}
J(X)\simeq \frac{(F^1H^2(X,\C))^\vee}{H_2(X,\Z)}
\end{equation}
where $(F^1H^2(X,\C))^\vee$ is the dual $\C$-linear space of $F^1H^2(X,\C)$ and we regard $H_2(X,\Z)$ as a subgroup of $(F^1H^2(X,\C))^\vee$ by integration. 

Let $\xi$ be a $(2,1)$-cycle represented by (\ref{formalsum}).
Let $D_j$ be the normalization of $C_j$ and $\mu_j: D_j\rightarrow X$ be the composition of $D_j\rightarrow C_j\hookrightarrow X$. 
We will define a topological 1-chain $\gamma_j$ on $D_j$.
If $f_j$ is constant, we define $\gamma_j = 0$.
If $f_j$ is not constant, we regard $f_j$ as a finite morphism from $D_j$ to $\P^1$. 
Then we define $\gamma_j$ as the pull-back of $[\infty, 0]$ by $f_j$, where $[\infty, 0]$ is a path on $\P^1$ from $\infty$ to $0$ along the positive real axis. 
By the condition $\sum_j \div_{C_j}(f_j) = 0$, $\gamma = \sum_j (\mu_j)_*\gamma_j$ is a topological 1-cycle on $X$.
Since $H_1(X, \Z) = 0$, there exists a 2-chain $\Gamma$ on $X$ such that $\partial\Gamma = \gamma$.
In this paper, $\Gamma$ is referred to as a \textit{2-chain associated with $\xi$}.
Then the image of $\xi$ under the regulator map is represented by the following element in $(F^1H^2(X,\C))^\vee$.
\begin{equation}\label{Levineformula}
\begin{tikzcd}
F^1H^2(X,\C)\ni \text{[}\omega\text{]} \arrow[r,mapsto] & \displaystyle\int_\Gamma\omega  + \sum_j\dfrac{1}{2\pi\sqrt{-1}}\displaystyle\int_{D_j-\gamma_j}\log (f_j)\mu_j^*\omega
\end{tikzcd}
\end{equation}
where $\log(f_j)$ is the pull-back of a logarithmic function on $\P^1-[\infty,0]$ by $f_j$.

From the formula, for a prime divisor $C$ on $X$ and $\alpha\in \C^\times$, the image of the $(2,1)$-cycle represented by $(C,\alpha)$ under the regulator map is $[C]\otimes \frac{1}{2\pi \sqrt{-1}}\log(\alpha)$.
In particular, we have 
\begin{equation}\label{JNS}
\nu(\CH^2(X,1)_\dec) = \NS(X)\otimes_\Z(\C/\Z) \quad (\subset  J(X)).
\end{equation}

\subsection{Normal functions}
We consider the regulator map in the relative setting.
Let $S$ be a complex manifold and $\mathcal H = (\mathcal H_\Z, F^\bullet)$ be a variation of $\Z$-Hodge structure of weight 2 over $S$.
Then 
\begin{equation*}
J(\mathcal H) = \frac{\mathcal H_\Z\otimes \O_S}{\mathcal H_\Z + F^2\mathcal H}
\end{equation*}
is a family of generalized complex tori over $S$.
Let $\pi: \X\rightarrow S$ be a smooth family of projective surfaces such that $H^1(\X_s,\Z)=0$ for every $s\in S$.
Let $\J(\X)\rightarrow S$ be the family of Jacobians attached to the variation of Hodge structure $R^2\pi_*\Z_\X$.
The fiber of $\J(\X)\rightarrow S$ over $s\in S$ is $J(\X_s)$.

Suppose that we have irreducible divisors $\Cc_j$ on $\X$ which are smooth over $S$ and non-zero meromorphic functions  $f_j$ on $\Cc_j$ whose zeros and poles are also smooth over $S$.
Assume that they satisfy the condition $\sum_j\div_{(\Cc_j)_s}((f_j)_s) = 0$ on each fiber $\X_s$.
Then we have a family of $(2,1)$-cycles $\xi = \{\xi_s\}_{s\in S}$ such that $\xi_s\in \CH^2(\X_s,1)$ is represented by the formal sum $\sum_{j}((\Cc_j)_s,(f_j)_s)$.
In this paper, such a family of $(2,1)$-cycles is called an \textit{analytic family of $(2,1)$-cycles}.
Then the section $S\rightarrow \J(\X); s\mapsto \nu(\xi_s)$ is holomorphic (\cite{CL} Proposition 4.1) and satisfies the horizontality condition, namely it is a normal function (\cite{CDKL16} Remark 2.1). 
We denote this section by $\nu(\xi)$ and call it the \textit{normal function associated with $\xi$}.

\section{Eisenstein Jacobians}
In this section, we define the Eisenstein Jacobian and prove its elementary property.
This plays a crucial role in the induction argument and detecting the non-trivial cycles.

Let $(X,\sigma)$ be an Eisenstein $K3$ surface.
The lattice $E(X,\sigma)$ is the kernel of $(\sigma^2)^*+\sigma^*+\id$, hence it is a sub Hodge structure of $H^2(X,\Z)$.
The \textit{Eisenstein Jacobian} $J(X,\sigma)$ is defined by
\begin{equation}
J(X,\sigma) = J(E(X,\sigma)^\vee),
\end{equation}
which is the Jacobian associated with the weight 2 Hodge structure\footnote{We regard the dual as a Hodge structure of weight 2 by Tate twists.
Hereafter we omit the notation of such Tate twists.} $E(X,\sigma)^\vee$, which is the dual of $E(X,\sigma)$.
Since $E(X,\sigma)$ has the $\Z[\zeta]$-module structure by $\sigma^*$-action, $J(X,\sigma)$ also has the $\Z[\zeta]$-module structure.

The inclusion $E(X,\sigma)\hookrightarrow H^2(X,\Z)$ induces the map 
\begin{equation}\label{surjHodge}
\begin{tikzcd}
H^2(X,\Z)=H^2(X,\Z)^\vee \arrow[r] & E(X,\sigma)^\vee,
\end{tikzcd}
\end{equation}
where the first equality follows from the unimodularity of $H^2(X,\Z)$.
Since (\ref{surjHodge}) is surjective after tensoring with $\Q$, this induces the surjection $J(X)\twoheadrightarrow J(X,\sigma)$.
Then we define the variant of the regulator map by the composition
\begin{equation*}
\begin{tikzcd}
\nu_{E}:\CH^2(X,1) \arrow[r,"\nu"] & J(X) \arrow[r,twoheadrightarrow] &J(X,\sigma).
\end{tikzcd}
\end{equation*}

Let $\pi: (\X,\sigma)\rightarrow S$ be a family of Eisenstein $K3$ surfaces, i.e., $\pi:\X\rightarrow S$ is a family of $K3$ surfaces and $\sigma: \X\rightarrow \X$ is an $S$-automorphism of order $3$ such that for every $s\in S$, $\sigma^*\omega_s = \zeta \omega_s$. 
We define the variation of Hodge structure $\mathcal E(\X,\sigma)$ as the kernel of $(\sigma^2)^*+\sigma^*+\id: R^2\pi_*\Z_\X\rightarrow R^2\pi_*\Z_\X$ and $\mathcal E(\X,\sigma)^\vee$ as its dual.
We have the morphism $R^2\pi_*\Z_{\X}\rightarrow  \mathcal E(\X,\sigma)^\vee$ as in (\ref{surjHodge}).
Let $\J(\X,\sigma) = \J(\mathcal E(\X,\sigma)^\vee)$ be the family of generalized complex tori over $S$ associated with $\mathcal E(\X,\sigma)^\vee$.
The fiber of $\J(\X,\sigma)\rightarrow S$ over $s\in S$ is $J(\X_s,\sigma_s)$.
The morphism $R^2\pi_*\Z_{\X}\rightarrow  \mathcal E(\X,\sigma)^\vee$ induces the surjection 
\begin{equation}\label{surjJac}
\begin{tikzcd}
\J(X)\arrow[r,twoheadrightarrow] & \J(\X,\sigma).
\end{tikzcd}
\end{equation}

If $(\X_s,\sigma_s)$ is type $(r,a)$ for all $s\in S$ (in the sense of Section 2.1), we have the natural morphism $S\rightarrow \M_{r,a}$ to the moduli space.

\begin{prop}\label{nonzeroprop}
Suppose that $S\rightarrow \M_{r,a}$ is dominant.
Then for a very general $s\in S$, $\nu_E$ factors $\CH^2(\X_s,1)_\ind$.
\end{prop}
\begin{proof}
It suffices to show that $\nu_{E}$ annihilates the decomposable part for a very general $s\in S$.
By (\ref{JNS}), the image of $\CH^2(\X_s,1)_\dec$ under the regulator map is $\NS(\X_s)\otimes_\Z (\C/\Z)$.
By Proposition \ref{NSprop} and the assumption that $S\rightarrow \M_{r,a}$ is dominant, this coincides with $L(\X_s,\sigma_s)\otimes_\Z (\C/\Z)$ for a very general $s\in S$.
Since $L(\X_s,\sigma_s)$ and $E(\X_s,\sigma_s)$ are orthogonal to each other, $L(\X_s,\sigma_s)\otimes_\Z (\C/\Z)$ is mapped to $0$ by the projection $J(\X_s)\twoheadrightarrow J(\X_s,\sigma_s)$.
Thus $\nu_{E}$ annihilates $\CH^2(\X_s,1)_\dec$ for a very general $s\in S$.

\end{proof}

\section{Construction of cycles}
In this section, we recall a construction of Eisenstein $K3$ surfaces as a triple covering of $\P^1\times \P^1$ and we construct $(2,1)$-cycles on such Eisenstein $K3$ surfaces.
Hereafter we use the notation $\O(n,m)$ which denotes the sheaf $\O_{\P^1\times \P^1}(n,m)$ on $\P^1\times \P^1$ for $n,m\in \Z$.

\subsection{A triple covering of $\P^1\times \P^1$}
Let $Y=\P^1\times \P^1$ and $B\in |\O(3,3)|$ be a divisor satisfying the following condition $(\mathbf{B})$.
\begin{equation*}
(\mathbf{B}):  \text{$B$ is reduced and has at worst nodes as singularities.}
\end{equation*}
Note that $B$ is not necessarily irreducible.
Then we define $\pi_1:\overline{X}\rightarrow X$ as the triple covering of $Y$ branching along $B$ and $\pi_2:X\rightarrow \overline{X}$ as the minimal resolution of singularities.
Let $\pi = \pi_1\circ \pi_2$.
\begin{equation}
\begin{tikzcd}
X \arrow[r,"\pi_2"] \arrow[rr,bend right = 20,"\pi"']& \overline{X} \arrow[r,"\pi_1"] & Y
\end{tikzcd}
\end{equation}
By the ramification formula, $X$ is a $K3$ surface.
The covering transformation on $\overline{X}$ lifts to an automorphism $\sigma\in \Aut(X)$ of order 3.
Since $Y$ is rational, the action of $\sigma$ on a non-zero holomorphic 2-form $\omega$ on $X$ is non-trivial.
Therefore, by replacing $\sigma$ by $\sigma^{-1}$ if necessary, we have an Eisenstein $K3$ surface $(X,\sigma)$.
We say that $(X,\sigma)$ is the Eisenstein $K3$ surface \textit{associated with $B$.}

By Proposition 4.6 in \cite{MOT}, we can calculate generators of the lattice $L(X,\sigma)$ of the Eisenstein $K3$ surface $(X,\sigma)$ associated with $B$ as follows.
\begin{prop}\label{genL}
For each node $q\in B$, let $\Lambda_q$ be the sublattice of $\NS(X)$ generated by two exceptional curves over the $A_2$-singular point on $\overline{X}$ over $q$.
For each irreducible component $B_i$ of $B$, let $D_i$ be the reduced curve on $X$ such that $\pi(D_i) = B_i$.
Then $L(X,\sigma)$ is generated by the sublattice $\pi^*(\NS(\P^1\times \P^1))\oplus \bigoplus_q \Lambda_q$ and classes of $D_i$.
\end{prop}

From this result, we can calculate the invariants $(r,a)$ of the Eisenstein $K3$ surface associated with $B$ as follows.
\begin{prop}\label{racomp}
Let $k$ be the number of irreducible components of $B$ and $n$ be the number of nodes on $B$.
Then $r=2n+2$ and $a=n+4-2k$.
\end{prop}
\begin{proof}
Since all $D_i$ are contained in $\pi^*(\NS(\P^1\times \P^1))\otimes \Q$, by Proposition \ref{genL}, $r=\rank L(X,\sigma)$ is equal to the rank of $\pi^*(\NS(\P^1\times \P^1))\oplus \bigoplus_q \Lambda_q$.
Since $\rank \Lambda_q=2$ for each $q$, we have $r=2n+2$.
Next, we will calculate the invariant $a$.
Fixed curves by the $\sigma$-action coincide with $D_i$, so the number of fixed curves coincides with $k$.
Furthermore, isolated fixed points by the $\sigma$-action coincide with the intersection points of two exceptional divisors over the nodes of $B$, so the number of isolated fixed points is $n$.
By \cite{AS08} pp.~912, we have the relation $a=n+4-2k$ between the topological invariants $k,n$ and the lattice invariant $a$.
Thus we have the result.
\end{proof}

\subsection{Construction of cycles}
Assume $B\in |\O(3,3)|$ satisfies the condition $(\mathbf{B})$.
Let $(X,\sigma)$ be the Eisenstein $K3$ surface associated with $B$.
Let $F_1\in |\O(1,0)|$ and $F_2\in |\O(0,1)|$ be divisors satisfying the following 3 conditions.
\begin{enumerate}
\renewcommand{\theenumi}{F\arabic{enumi}}
\item There exists a point $q_1\in B\cap F_1$ such that the intersection multiplicity of $B$ and $F_1$ at $q_1$ is $2$.
\item There exists a point $q_2\in B\cap F_2$ such that the intersection multiplicity of $B$ and $F_2$ at $q_2$ is $2$.
\item The points $q_1$ in (i) and $q_2$ in (ii) are distinct.
\end{enumerate}
The points $q_1,q_2$ can be a nodes of $B$ and is uniquely determined by $F_i$ and for each $B$.
If $h(x,y)$ is a local equation of $B$, $F_1$ is given by $x=\alpha$ where $\alpha$ is a root of $\mathrm{Res}_y\left(h,\dfrac{\del h}{\del y}\right)$ with multiplicity at most 2 and $F_2$ is given by $y=\beta$ where $\beta$ is a root of $\mathrm{Res}_x\left(h,\dfrac{\del h}{\del x}\right)$ with multiplicity at most 2.
In particular, there exists the finite number of the choices of $F_1$ and $F_2$ for each $B$.
For $i=1,2$, $F_i$ is called \textit{a tangent type} if $q_i$ is a smooth point of $B$ and \textit{a node-passing type} if $q_i$ is a node of $B$.
If $x=\alpha$ (resp. $y=\beta$) is a simple root of the resultant, $F_1$ (resp. $F_2$) is a tangent type and otherwise $F_1$ (resp. $F_2$) is a node-passing type.
The typical situation of $(B,F_1,F_2)$ is illustrated in Figure \ref{typ}.
In this figure, $F_1$ is a node-passing type and $F_2$ is a tangent type.

\begin{figure}
\includegraphics[width = 70mm]{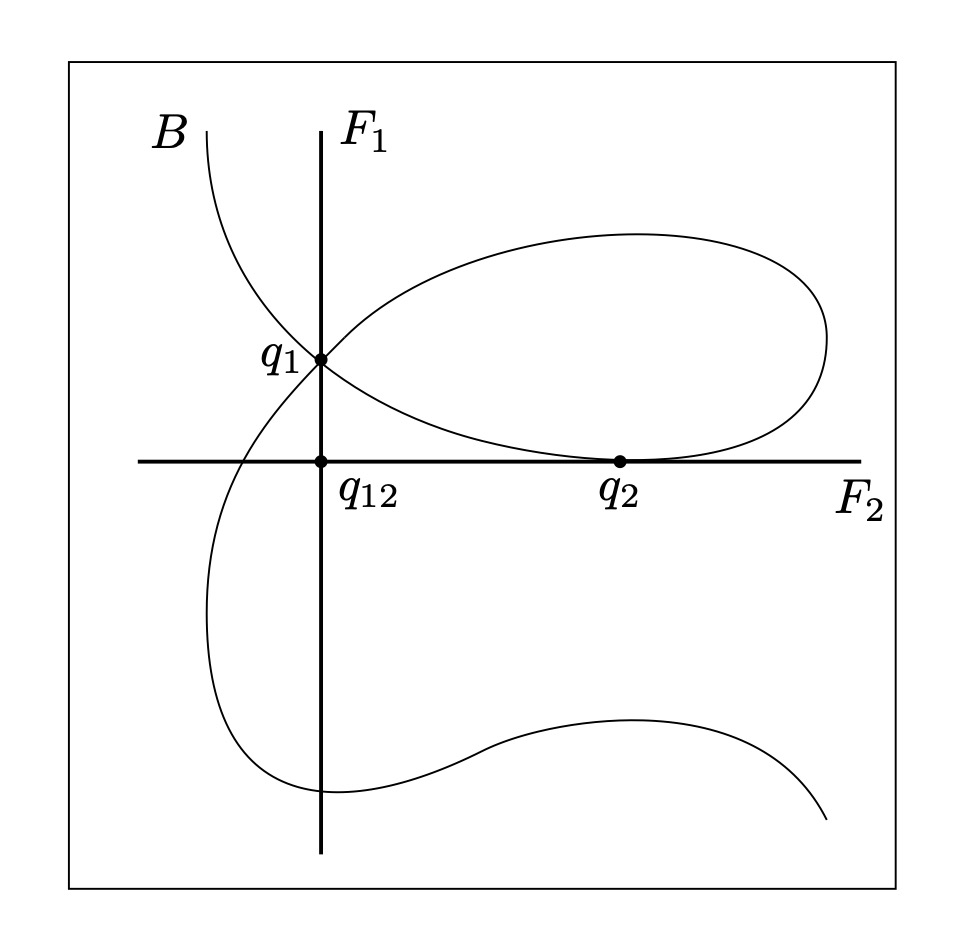}
\caption{The typical situation of $F_1,F_2$}
\label{typ}
\end{figure}

For $i=1,2$, let $C_i\subset X$ be the strict transformation of $\pi_1^{-1}(F_i)$.
By the condition (F1) and (F2), $C_i\rightarrow F_i$ is a triple covering ramified at $2$ points, hence $C_i$ is a rational curve.
If $F_i$ is a node-passing type, the cusp on $\pi^{-1}_1(F_i)$ over $q_i$ is blown-up by $X\rightarrow \overline{X}$, thus $C_i$ is smooth.
On the other hand, if $F_i$ is a tangent type, $C_i$ has the cusp over $q_i$.

Let $q_{12}$ be the unique intersection point of $F_1$ and $F_2$.
By the condition (F3), $q_{12}\not \in B$, so $\pi^{-1}(q_{12})$ consists of $3$ points.
We choose a point $p_0\in \pi^{-1}(q_{12})$ and set $p_1 = \sigma(p_0)$ and $p_2 = \sigma(p_1)$.
Since $C_1$ and $C_2$ are rational curves and $p_0,p_1$ are points on them, we can find rational functions $f_{1}\in \C(C_1)^\times$ and $f_{2}\in \C(C_2)^\times$ such that 
\begin{equation}\label{divrelation}
\div_{C_1}(f_{1}) = -\div_{C_2}(f_{2}) = p_1-p_0.
\end{equation}

\begin{defn}
Let $\xi$ be the $(2,1)$-cycle on $X$ represented by the formal sum
\begin{equation}
(C_1,f_{1}) + (C_2,f_{2}).
\end{equation}
Let $\widetilde{\Xi} =\{\xi,\sigma_*(\xi),\sigma^2_*(\xi)\}\subset \CH^2(X,1)$.
The subset $\widetilde{\Xi}$ does not depend on the choice of $p_0\in \pi^{-1}(q_{12})$.
Let $\Xi = \nu_E(\widetilde{\Xi})\subset J(X,\sigma)$.

We have freedom of the choices of $f_{1}$ and $f_{2}$, but the ambiguity is only the multiplication by some nonzero constant functions.
If we take different rational functions satisfying (\ref{divrelation}), the resulting $(2,1)$-cycles differs by decomposable cycles of the form $(C_1,\alpha_1)+(C_2,\alpha_2)$, where $\alpha_1,\alpha_2\in \C^\times$.
By Proposition \ref{genL}, $[C_1],[C_2]\in L(X,\sigma)$, so the image of them under $\nu_{E}$ is $0$.
Therefore, $\Xi$ is not affected by such ambiguity.
\end{defn}

As we see above, for $i=1,2$, $C_i$ has a cusp $\hat{q_i}\in C_i$ over $q_i$ if $F_i$ is a tangent type.
Therefore, in this case, $\xi$ does not form an \textit{analytic} family of $(2,1)$-cycles in the sense of this paper.
To overcome this, we consider the blowing-up $b: \widetilde{X}\rightarrow X$ along the union of points $\{\hat{q_i}: \text{$F_i$ is a tangent type. ($i=1,2$)}\}$.
If we denote the strict transformation of $C_i$ by $\widetilde{C_i}$, the formal sum $(\widetilde{C}_1,b^*(f_{1})) + (\widetilde{C}_2,b^*(f_{2}))$ represents the $(2,1)$-cycle $b^*(\xi)$.
Since $\widetilde{C}_i$ is smooth, $b^*(\xi)$ is represented by rational functions on smooth curves.

Since the morphism $b_*\circ b^*:H^2(X,\Z)\rightarrow H^2(X,\Z)$ is the identity map, we have $\nu(\xi) = b_*(\nu(b^*(\xi)))$.
Thus $\Xi$ is the image of $\nu(b^*(\xi)), \nu(b^*(\sigma_*(\xi)))$ and $\nu(b^*(\sigma^2_*(\xi)))$ under the map $J(\widetilde{X})\rightarrow J(X)\rightarrow J(X,\sigma)$.

\section{The parameter space $\hat{U}_r$}
In this section, we define the family of Eisenstein $K3$ surfaces and $(2,1)$-cycles over a parameter space $\hat{U}_r$, which is a dominant family of type $\M_{r,a}$.

We define the following locally closed subsets $U_r$ of $|\O(3,3)|$ for $r=2,4,\dots, 18$.
In the following, $B_{a,b}$ and $B'_{a,b}$ denote reduced and irreducible curves of bidegree $(a,b)$.
(See Figure. \ref{branchlocus}.)
\begin{equation*}
\begin{aligned}
U_{2} &= \left\{B_{3,3} \: | \:B_{3,3}\text{ is smooth.}\right\} \\
U_{4} &= \left\{B_{3,3}  \:| \:B_{3,3}\text{ has a node.}\right\} \\
U_{6} &= \left\{B_{3,3}  \:| \:B_{3,3}\text{ has 2 nodes.}\right\} \\
U_{8} &= \left\{B_{3,3}  \:| \:B_{3,3}\text{ has 3 nodes.}\right\} \\
U_{10} &= \left\{B_{3,3}  \:| \:B_{3,3}\text{ has 4 nodes.}\right\} \\
\end{aligned}
\end{equation*}
\begin{equation*}
\begin{aligned}
U_{12} &=  \left\{B_{2,1}+B_{1,2} \:|\: B_{2,1}\text{ and }B_{1,2} \text{ are smooth and their sum satisfies $(\textbf{B})$.} \right\}\\
U_{14} &=  \left\{B_{1,2}+B_{1,1}+B_{1,0} \:|\:\text{Each of $B_{a,b}$ is smooth and their sum satisfies $(\textbf{B})$.} \right\}\\
U_{16} &= \left\{B_{1,1}+B'_{1,1}+B_{1,0}+B_{0,1} \:\left|\: 
\begin{aligned}
&\text{Each of $B_{a,b}$ and $B'_{a,b}$ is smooth}\\
&\text{and their sum satisfies $(\textbf{B})$.}
\end{aligned}
\right. \right\}\\
U_{18} &= \left\{B_{1,1}+B_{1,0}+B'_{1,0}+B_{0,1}+B'_{0,1} \:\left|\: 
\begin{aligned}
&\text{Each of $B_{a,b}$ and $B'_{a,b}$ is smooth}\\
&\text{and their sum satisfies $(\textbf{B})$.}
\end{aligned}
\right. \right\}
\end{aligned}
\end{equation*}
For each $B \in U_r$, we can construct an Eisenstein $K3$ surface $(X,\sigma)$ associated with $B$ by the construction.
Furthermore, we can construct a family of Eisenstein $K3$ surfaces $(\X ,\sigma) \rightarrow U_r$ if we shrink $U_r$ to an Zariski open subset.

\begin{figure}

\begin{minipage}{.45\linewidth}

\includegraphics[width = 55mm]{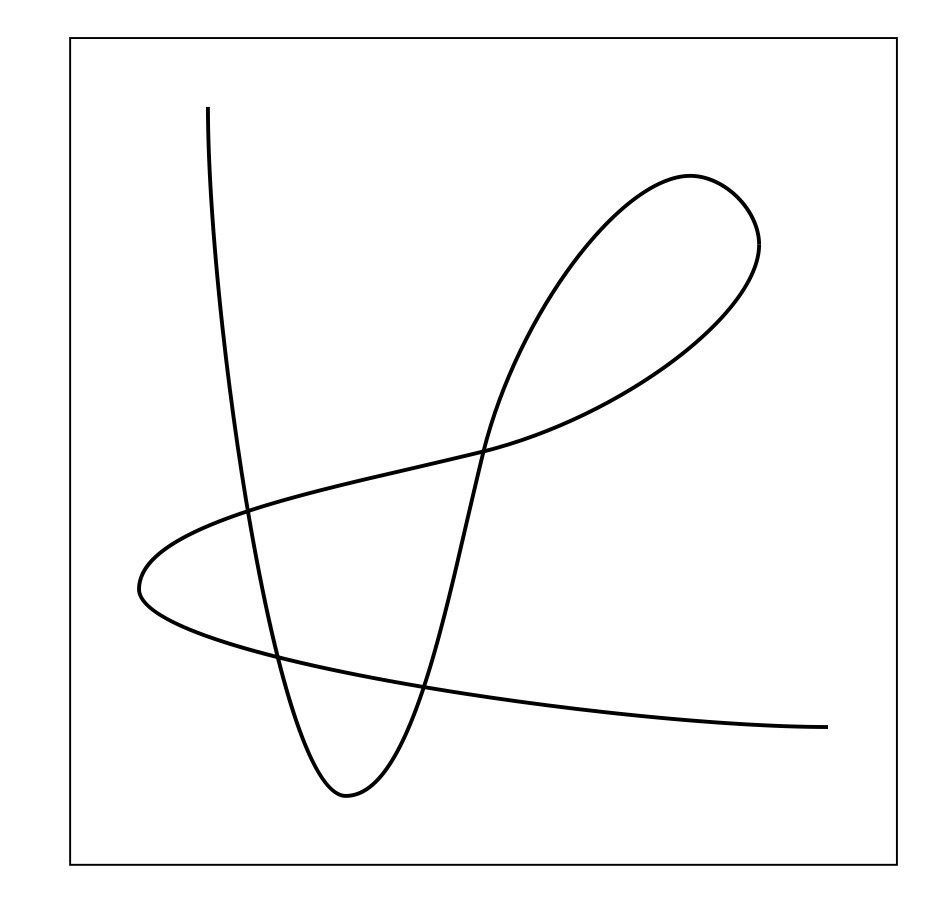}
\caption*{$r=10$}
\end{minipage}
\begin{minipage}{.45\linewidth}
\includegraphics[width = 55mm]{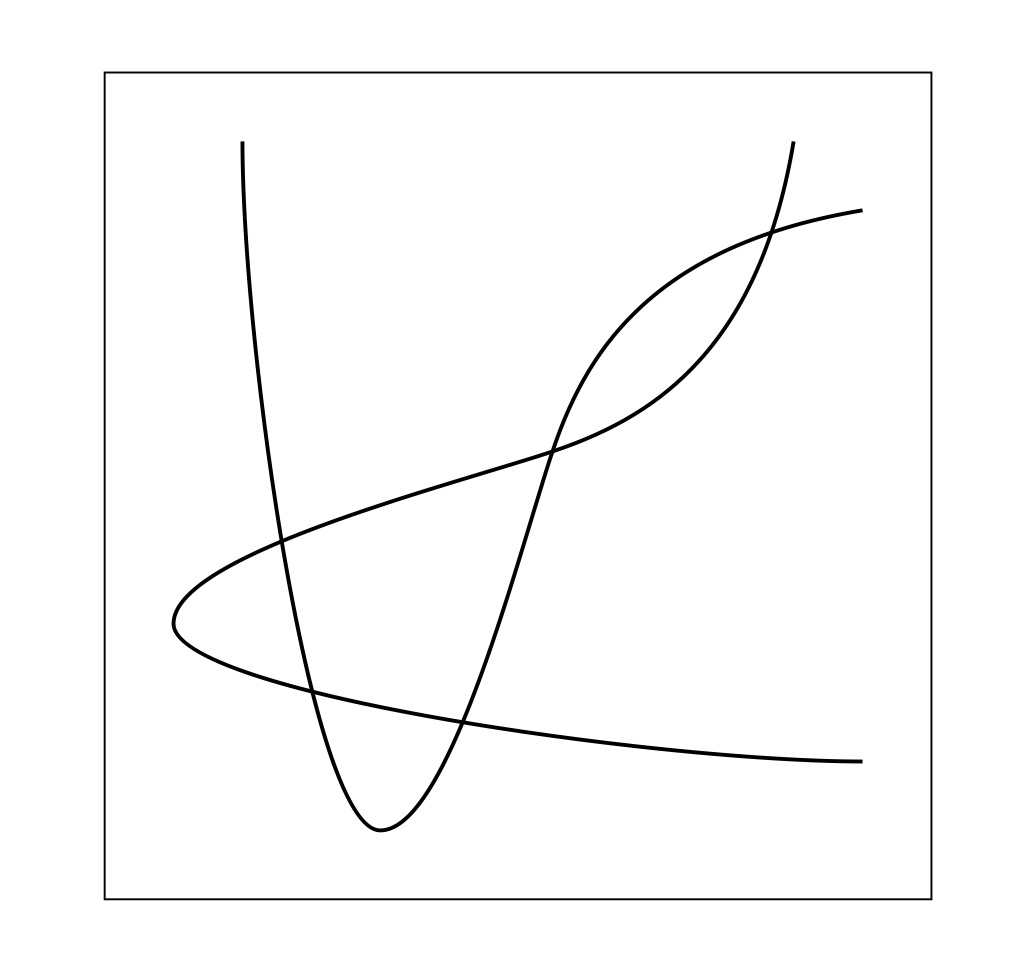}
\caption*{$r=12$}
\end{minipage}
\begin{minipage}{.45\linewidth}
\includegraphics[width = 55mm]{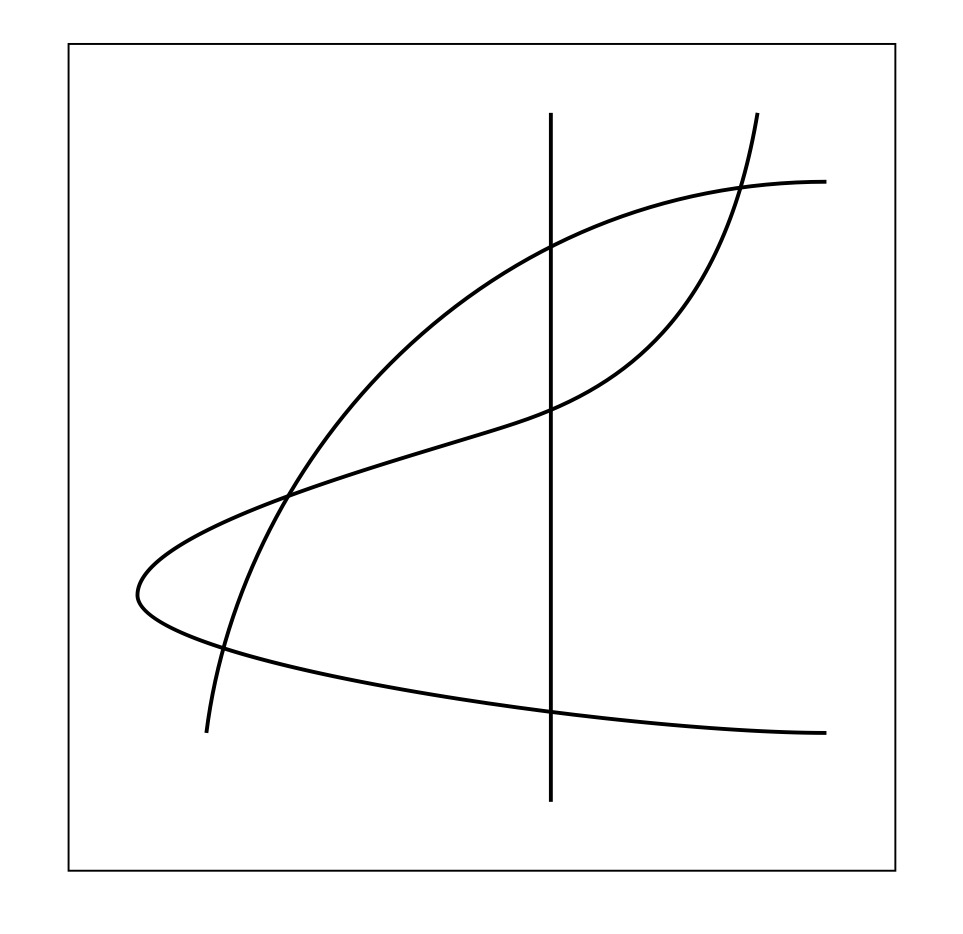}
\caption*{$r=14$}
\end{minipage}
\begin{minipage}{.45\linewidth}
\includegraphics[width = 55mm]{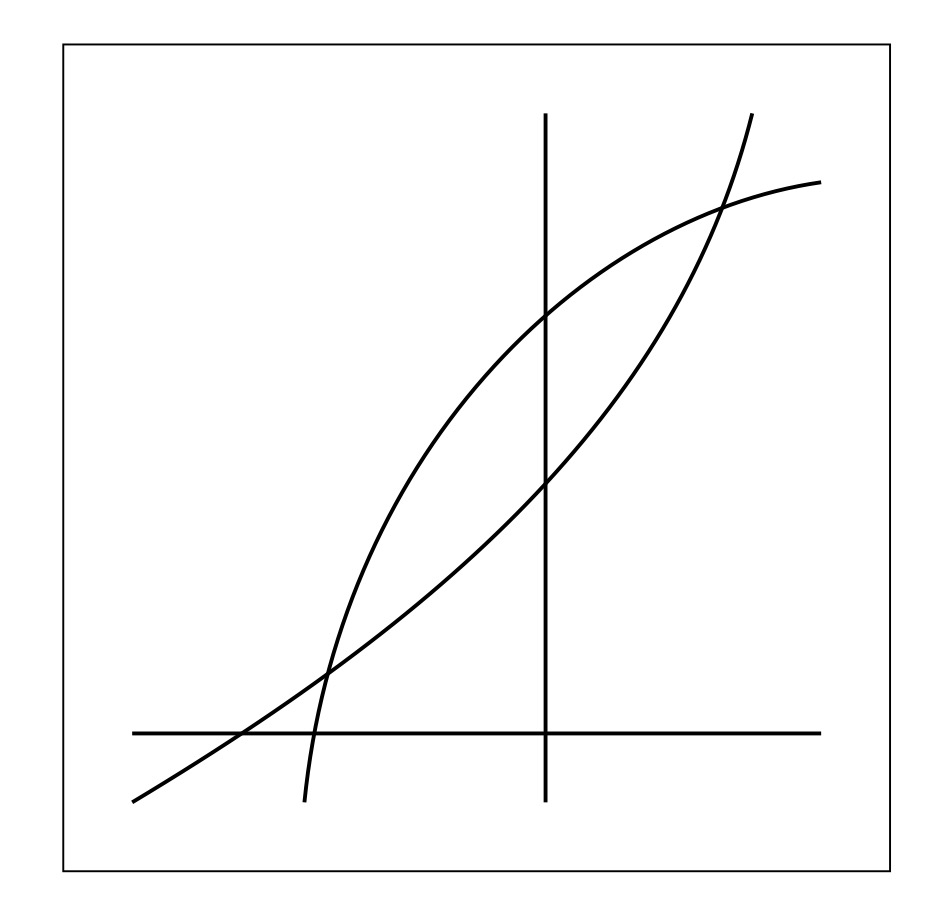}
\caption*{$r=16$}
\end{minipage}

\centering

\caption{Branching curvess corresponding to $r=10,12,14,16$}
\label{branchlocus}
\end{figure}

The precise construction of $\X \rightarrow U_r$ is as follows.
Let $\Y = \P^1\times \P^1\times U_r$.
By removing a hyperplane section from $U_r$ if necessary, we can take a global section $s\in \Gamma(\Y,\O_{\Y}(3,3))$ such that for each $B \in U_r$, $\div(s_B) = B$.
Then by taking the triple covering associated with $(\O_\Y(-1,-1),s)$, we have a $3:1$ map $\overline{\X} \rightarrow \Y$.
Since the number of nodes of $B$ is constant, the singular locus $\Zc$ of $\overline{\X}$ is flat over $U_r$.
By blowing-up $\overline{\X}$ along $\Zc$, we have a (smooth) family of Eisenstein $K3$ surfaces $(\X,\sigma)\rightarrow U_r$.

\begin{prop}\label{dominate}
For each $r=2,4,\dots, 18$, we set $a$ be $1+r/2$ when $r\le 10$ and $11-r/2$ when $r\ge 12$.
For each $B\in U_r$, the Eisenstein $K3$ surface associated with $B$ is type $(r,a)$.
Furthermore, the natural map $U_r\rightarrow \M_{r,a}$ is dominant.
\end{prop}
\begin{proof}
By calculating of the intersection numbers, the number $n$ of nodes of $B\in U_r$ is given by $n=(r-2)/2$.
Then we can check the former statement by the formula in Proposition \ref{racomp}.
For the latter statement, by the similar arguments in \cite{MOT} \S4.3, it is enough to check that the dimensions of $\M_{r,a}$ and $U_r//(\mathrm{PGL}_2)^2$ coincide.
By taking a suitable coordinates, we can check the coincidence of the dimensions by examining equations in each case.
\end{proof}
The case $r=2$ of Proposition \ref{dominate} is proved in \cite{Kon}, and $r=4$ is proved in Example 4.10 in \cite{MOT}.

Let $\hat{U}$ be the locally closed subset of $|\O(3,3)|\times |\O(1,0)|\times |\O(0,1)|$ defined by
\begin{equation}
\hat{U} =\left\{(B,F_1,F_2) : \begin{aligned}
&\text{$B$ satisfies the condition $(\mathbf{B})$, and }\\
&\text{$(B,F_1,F_2)$ satisfies the condition (F1)(F2)(F3).}
\end{aligned}
\right\}.
\end{equation}
We have a natural projection $\hat{U}\rightarrow |\O(3,3)|$.
For each $r=2,4,\cdots, 18$, let $\hat{U}_r$ be the inverse image of $U_r$ by $\hat{U}\rightarrow |\O(3,3)|$.
The morphism $\hat{U}_r\rightarrow U_r$ is finite, \'etale and surjective. 
By the induction arguments in the next section, the following is important.
\begin{prop}\label{closure}
For each $2\le r\le 16$, $\hat{U}_{r+2}$ is contained in the Zariski closure of $\hat{U}_r$.
\end{prop}
\begin{proof}
It is enough to check that $U_{r+2}$ is contained in the Zariski closure of $U_r$.
Only $U_{12}\subset \overline{U_{10}}$ is the non-trivial case, so we prove this.
Let $B\in U_{12}$. 
By blowing up along a node of $B$ and contracting two fibers passing through the node, $B$ can be transformed to a union of 2 smooth conics $B_1$ and $B_2$ in $\P^2$ intersecting transversally.
Thus it is enough to check that $B$ can be deformed into quartics with 3 nodes.
This can be checked by fixing 3 of intersection points of $B_1$ and $B_2$ at $[1:0:0], [0:1:0],[0:0:1]$ in $\P^2$.
\end{proof}

Let $\X\rightarrow \hat{U}_r$ be the base change of $\X\rightarrow U_r$ by $\hat{U}_r\rightarrow U_r$.
For each point $u\in \hat{U}_r$, by the construction in \S4.2, we have a subset $\widetilde{\Xi}_u \subset \CH^2(\hat{\X}_u, 1)$ and $\Xi_u \subset J(\X_u,\sigma)$.
Let $\Xi = \{\Xi_u\}_{u\in \hat{U}_r} \subset \J(\X,\sigma)$.
Then $\Xi$ is a degree 3 holomorphic multisection of  $\J(\X,\sigma)\rightarrow \hat{U}_r$ in the following sense.
This enables us to consider the degeneration of $\Xi$.
As we will see in Section 7, $\Xi$ splits in some case, but we do not know whether $\Xi$ splits in general. 

\begin{prop}\label{multisection}
For each $u\in \hat{U}_r$, there exits an open neighborhood $V$ (in the sense of the classical topology) of $u$ and holomorphic sections $\nu_i:V\rightarrow \J(\X,\sigma)|_V$ for $i=0,1,2$ such that $\Xi|_V=\{\nu_0,\nu_1,\nu_2\}$.
\end{prop}
\begin{proof}
Let $u\in \hat{U}_r$.
Let $\Y = \P^1\times \P^1\times \hat{U}_r$.
For $i=1,2$, let $\mathcal F_i\subset \Y$ be the family of curves corresponding to $F_i$ and $\Cc_i\subset \X$ be the strict transformation of the pull-back of $\mathcal F_i$ by $\overline{\X}\rightarrow \Y$.
Then $\Cc_i$ be a flat family of rational curves over $\hat{U}_r$, and $\Cc_1\cap \Cc_2$ is a $\Z/3\Z$-torsor over $\hat{U}_r$ in the \'etale topology, hence there exists an open neighborhood $V$ of $u$ in the sense of classical topology such that we can take sections $p_0,p_1,p_2: V\rightarrow\Cc_1\cap \Cc_2$.

First, we consider the case when both of $(\mathcal F_1)_u$ and $(\mathcal F_2)_u$ are node-passing types for each $v\in V$.
Since each $B\in U_r$ has same number of nodes, both of $(\mathcal F_1)_v$ and $(\mathcal F_2)_v$ are also node-passing types for all $v\in V$.
Then both of $(\Cc_1)_v$ and $(\Cc_2)_v$ are smooth, so $\Cc_1|_V$ and $\Cc_2|_V$ are smooth family of rational curves over $V$ with 3 sections $p_0,p_1,p_2$.
Thus $\Cc_1|_V$ and $\Cc_2|_V$ are trivial $\P^1$-bundles over $V$, so we can take a meromorphic function $f_{1}$ (resp. $f_{2}$) on $\Cc_1|_V$ (resp. $\Cc_2|_V$) such that $\div_{(\Cc_1)_v}((f_{1})_v)=-\div_{(\Cc_2)_v}((f_{2})) = (p_0)_v-(p_1)_v$.
Then we can construct an analytic family of $(2,1)$-cycle $\xi$ on $V$.
The sections $\nu_i=\nu_E(\sigma_*^i\xi)\:\:(i=0,1,2)$ satisfy the condition.

Next, we consider the case when either $(\mathcal F_1)_u$ or $(\mathcal F_2)_u$ is a tangent type.
Let $I=\{i\in\{1,2\}: \mathcal (F_i)_u\text{ is a tangent type.}\}$.
Then for all $v\in V$ and $i\in I$, $(\mathcal F_i)_v$ is a tangent type.
For $i\in I$, by shrinking $V$ if necessary, we may assume that we have a section $q_i:V\rightarrow {\mathcal F}_i|_V$ such that for each $v\in V$, $q_i(v)$ is the point on which $({\mathcal F}_i)_v$ intersects with the branching locus with multiplicity 2.
Let $\hat{q}_i: V\rightarrow \Cc_i$ be the unique lift of $q_i$.
Let $b:\widetilde{\X}\rightarrow \X|_V$ be the blowing-up along $\bigcup_{i\in I} \hat{q}_i(V)$ and $\widetilde{\Cc}_i$ be the strict transformation of $\Cc_i$ for $i=1,2$.
Then $\widetilde{\Cc}_1$ and $\widetilde{\Cc}_1$ are smooth family of rational curves over $V$ with 3 sections.
Thus we can construct an analytic family of $(2,1)$-cycle $\xi$ on $\widetilde{\X}\rightarrow V$.
The section $\nu_i=b_*(\nu_E(\sigma_*^i\xi))\:\:(i=0,1,2)$ satisfy the condition.
\end{proof}

\section{The main theorem}
In this section, we state the main theorem and prepare the induction steps which reduces the proof for the general case to the starting case $r=18$.
The starting case $r=18$ will be proved in the next section.
\subsection{A degeneration lemma}
As in \cite{MS23}, our induction argument is based on the following degeneration lemma.
The proof is similar to \cite{MS23} Lemma 5.2, but in this paper, we consider somewhat more complicated degenerations, so we need more specific examination on the simultaneous resolution of $A_2$-singularities.

\begin{lem}\label{degenerationlemma}
Let $\Delta$ be the unit disk and $\Delta^*=\Delta-\{0\}$.
Assume a holomorphic map $\rho:\Delta \rightarrow |\O(3,3)|\times |\O(1,0)|\times |\O(0,1)|$ satisfying the following conditions is given.
\begin{enumerate}
\renewcommand{\theenumi}{\roman{enumi}}
\item For each $t\in \Delta$, $\rho(t)=(B(t),F_1(t),F_2(t))\in \hat{U}$.
\item There exists $r$ such that $\rho(\Delta^*)\subset \hat{U}_r$.
\end{enumerate}
By the condition (i), we have the Eisenstein $K3$ surface $(X_t,\sigma)$ associated with $B(t)$ and $\Xi_{t}\subset J(X_{t},\sigma)$ for each $t\in \Delta$.
Suppose that $\Xi_{0}$ is non-torsion\footnote{
A subset $A$ of an abelian group is called \text{non-torsion} if $m\cdot A\neq \{0\}$ for any positive integer $m$.}.
Then for a very general  $t\in \Delta$, $\Xi_{t}$ is non-torsion.
\end{lem}
\begin{proof}
First, we construct a family of $K3$ surfaces $\X\rightarrow \Delta$ such that each fiber corresponding to $X_t$ in the statement of this lemma.
Let $\Y=\P^1\times \P^1\times \Delta$ and $s\in \Gamma(\Y,\O_{\Y}(3,3))$ such that $\div_{\Y_t}(s_t)=B(t)$ for each $t\in \Delta$.
Then let $\pi_1: \overline{\X}'\rightarrow \Y$ be the triple covering associated with $(\O_\Y(-1,-1),s)$.
By the condition (ii), the number of nodes of $B(t)$ are constant on $t\in \Delta^*$.
Then by taking the cyclic base change $\Delta \rightarrow \Delta; t\mapsto t^N$ for some $N$ if necessary (to annihilate the monodromy), we have sections $z_1,z_2,\dots, z_n:\Delta \rightarrow \Y$ such that $z_1(t),z_2(t),\dots,z_n(t)$ corresponds to the nodes of $B(t)$ for $t\in \Delta^*$.
Note that $z_1(0),z_2(0),\dots, z_n(0)$ are also nodes of $B(0)$, but $B(0)$ may have other nodes $x_1,\dots, x_m$.

Let $\hat{z}_i:\Delta \rightarrow \overline{\X}'$ be the unique lift of $z_i$.
Let $\X'\rightarrow \overline{\X}'$ be the blowing up along $\bigcup_{i=1}^n\hat{z}_i(\Delta)$.
By taking the cyclic base change again, we may assume that for $i=1,2,\dots, n$, we have the families $\mathcal E_{i,1}$ and $\mathcal E_{i,2}$ of exceptional curves over $\hat{z}_i(t)$.

Let $\hat{x}_1,\hat{x}_2,\dots, \hat{x}_m$ be the points of the fiber $\X'_0$ which is the unique lifts of the nodes $x_1,\dots, x_m\in B(0)$.
Then $\hat{x}_1,\hat{x}_2,\dots, \hat{x}_m$ is the $A_2$ singular points of $\X'_0$.
By the result of \cite{Tju}, if we take the cyclic base change, we have the simultaneous minimal resolution $\X\rightarrow \X'$.
Then the resulting the family $p:\X\rightarrow \Delta$ is a smooth family of $K3$ surfaces.
The families constructed above are summarized in the following diagram.
\begin{equation}
\begin{tikzcd}
\X \arrow[r]\arrow[d,"p"']  
& \X' \arrow[r]\arrow[d] & \overline{\X}'\arrow[r,"\pi_1"]\arrow[d] &\Y \arrow[r, equal]\arrow[d] &[-16pt] \P^1\times \P^1\times \Delta \\
\Delta \arrow[r] & \Delta \arrow[r]  & \Delta \arrow[r,equal] & \Delta \\[-15pt]
t \aru \arrow[rrr,mapsto] &&& t^N \aru
\end{tikzcd}
\end{equation}
By the construction, for $t\in \Delta$, the fiber $\X_t$ corresponds to the Eisenstein $K3$ surface $(X_{t^N},\sigma)$ appearing in the statement of this lemma.

Let $\L\subset R^2p_*\Z_{\X}$ be the sub VHS generated by families of curves $\E_{i,j}\:\:(i\in \{1,2,\dots, n\}, j\in \{1,2\})$ and the pull-backs of generators of $\NS(\P^1\times \P^1)$ by the natural map $\X\rightarrow \Y\rightarrow \P^1\times \P^1$.
For each $t\in \Delta^*$, the $\Q$-linear space $L(X_{t^N},\sigma)_\Q$ is generated by $\pi^*\NS(\P^1\times \P^1)$ and exceptional curves over the nodes $z_i(t)$ by Proposition \ref{genL}, so we have 
\begin{equation}\label{LLcor}
(\L_{t})\otimes \Q = L(\X_{t},\sigma)\otimes \Q = L(X_{t^N},\sigma)\otimes \Q.
\end{equation}
Let $\J(\E)$ be the Jacobian associated with the variation of Hodge structure $\mathrm{Coker} (\L\rightarrow R^2p_*\Z_{\X})$.
We have the natural surjection from the family of Jacobians $\J(\X)$ to $\J(\E)$.
By (\ref{LLcor}), for each $t\in \Delta^*$, we have $\J(\E)_{t} = J(X_{t^N},\sigma)$.
Furthermore, since $\L_0 \subset L(X_0,\sigma)$ by Proposition \ref{genL}, we have a natural surjection 
\begin{equation}\label{surjspecial}
\J(\E)_0\twoheadrightarrow J(X_0,\sigma).
\end{equation}

Finally, we will construct a family of a holomorphic section $\nu: \Delta \rightarrow \J(\E)$ satisfying the following conditions.
\begin{enumerate}
\renewcommand{\theenumi}{\alph{enumi}}
\item For $t\in \Delta^*$, under the identification $\J(\E)_{t} = J(X_{t^N},\sigma)$, $\Xi_{t^N}\subset J(X_{t^N},\sigma)$ coincides with the set $\{\nu(t),\sigma_*(\nu(t)), \sigma^2_*(\nu(t))\}$.
\item $\Xi_0\subset J(X_0,\sigma)$ coincides with the set $\{\overline{\nu(0)}, \sigma_*(\overline{\nu(0)}), \sigma^2_*(\overline{\nu(0)})\}$ where $\overline{\nu(0)}$ is the image of $\nu(0)$ under the surjection (\ref{surjspecial}).
\end{enumerate}
If we can show the existence of such $\nu$, by the property (b) and the assumption $\Xi_0$ is non-torsion, $\nu$ is non-torsion section.
Thus $\nu(t)$ is non-torsion for a very general $t\in \Delta$, so by the property (a), we can show that $\Xi_t$ is non-torsion for a very general $t\in\Delta^*$.

For $i = 1,2$, let $\mathcal F_i = \{F_i(t)\}_{t\in \Delta}$ be the family of rational curves on $\Delta$ and $q_i:\Delta \rightarrow \Y$ be the section such that $B(t)$ and $F_i(t)$ intersect at $q_i(t)$ by multiplicity $2$.
Let $q'_i:\Delta \rightarrow \overline{\X}'$ be the unique lift of $q_i$.
For $i=1,2$, we define a closed subvariety $\mathcal{W}_i$ of $\X$ as follows.

\begin{figure}

\includegraphics[width = 110mm]{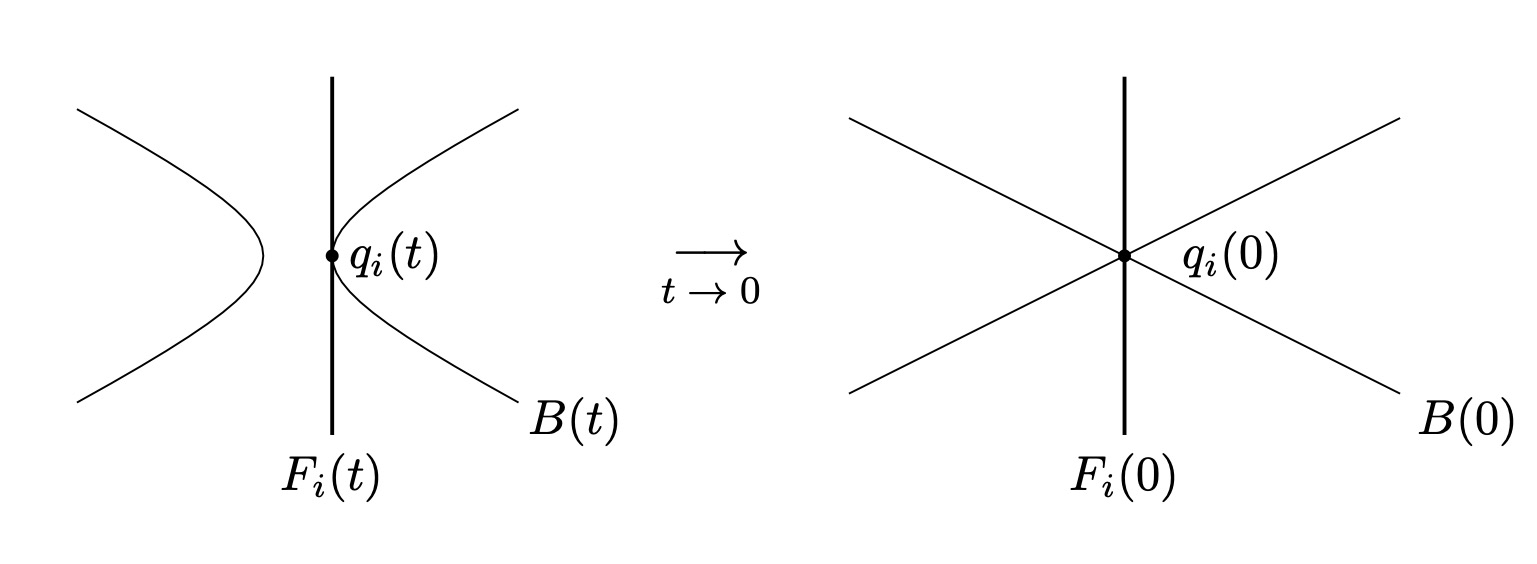}
\caption{Figure for Case 3}
\end{figure}

\noindent
\underline{Case 1.} $F_i(0)$ is a tangent type.\par
In this case, $F_i(t)$ is a tangent type for all $t\in \Delta$.
Since $q'_i(t)$ does not intersect with the blowing-up locus for each $t \in \Delta$, we have the unique lift $\hat{q}_i:\Delta\rightarrow \X$ of $q'_i$, and we set $\mathcal{W}_i = \hat{q}_i(\Delta)$.

\noindent
\underline{Case 2.} $F_i(t)$ is a node-passing type for some $t\in \Delta^*$.\par
Then $q_i$ corresponds to either $z_1,z_2,\dots, z_n$ above and we set $\mathcal{W}_i = \emptyset$.

\noindent
\underline{Case 3.} $F_i(0)$ is a node-passing type and $F_i(t)$ is a tangent type for $t\in \Delta^*$.\par
We will construct a lift $\hat{q}_i:\Delta\rightarrow \X$ of $q'_i$ and we set $\mathcal{W}_i= \hat{q}_i(\Delta)$.
Let $\mathcal{V}\rightarrow \Delta$ be a sufficiently small neighborhood (in the sense of classical topology) of $q_i:\Delta \rightarrow \Y$.
Then we may assume that we have a biholomorphic map $\mathcal{V}\simeq \Delta^2\times \Delta$ over $\Delta$ inducing the following isomorphism.
\begin{equation}
\begin{aligned}
&B(t)\cap \mathcal{V} \simeq \{(x,y,t)\in \Delta^2\times \Delta: xy+t^2=0\},\\
&F_i(t) \cap \mathcal{V} \simeq \{(x,y,t)\in \Delta^2\times \Delta: y=x+2t\}
\end{aligned}
\end{equation}
Under these isomorphism, $q_i(t)$ corresponds to the point $(x,y,t)=(-t,t,t)$.
Let $\mathcal{U}'$ (resp. $\mathcal{U}$) be the pull-back of $\mathcal{V}$ by $\X'\rightarrow \Y$ (resp. $\X\rightarrow \Y$).
Then $\mathcal{U}'$ is isomorphic to
\begin{equation}
\{(z,x,y,t)\in \C\times \Delta^2\times \Delta: z^3=xy+t^{2N}\}
\end{equation}
over $\Delta$ for some $N$.
By the explicit form of simultaneous resolution given in \cite{Tju} \S3, after a cyclic base change $\Delta\rightarrow \Delta; t\mapsto t^3$, $\mathcal{U}$ is given by the graph of the following meromophic map.
\begin{equation}
\begin{tikzcd}
\mathcal{U}' \arrow[r, dashrightarrow] & \P^1\times \P^1 &[-16pt]\\[-10pt]
(z,x,y,t) \aru \arrow[r,mapsto] &  ([w_{00}:w_{01}], [w_{10}:w_{11}]) \aru \arrow[r,equal] &
([x:z-t^{2N}],[x:(z-t^{2N})(z-\zeta t^{2N})])
\end{tikzcd}
\end{equation}
Then the holomorphic map $\hat{q}_i: \Delta \rightarrow \mathcal{U}; t\mapsto (z,x,y,t,[w_{00}:w_{01}], [w_{10}:w_{11}]) = (0,-t^{3N},t^{3N},[t^N:1],[1:-\zeta t^N])$ is the lift of $q'_i$, so we have a desired section and finish Case 3.

Let $b:\widetilde{\X}\rightarrow \X$ be the blowing-up along $\mathcal{W}_1\cup \mathcal{W}_2$ and $\widetilde{\Cc}_i \subset \widetilde{\X}$ be the strict transformation of $\pi_1^{-1}(\mathcal{F}_i)$ by $\widetilde{\X}\xrightarrow{b}\X\rightarrow \overline{\X}'$.
Then by the construction, $\widetilde{\Cc}_1$ and $\widetilde{\Cc}_2$ are smooth families of rational curves over $\Delta$.
Since $\widetilde{\Cc}_1\cap \widetilde{\Cc}_2$ is a $\Z/3\Z$-torsor over $\Delta$, we can construct an analytic family of higher Chow cycles $\xi_0$ on $\widetilde{\X}$ as in Proposition \ref{multisection}.
Let $\nu$ be the image of $\nu(\xi_0)$ under the surjection $\J(\widetilde{\X})\rightarrow \J(\X)\rightarrow \J(\E)$.
Since $b_*(\nu(\xi_0)_t)$ coincides with a $(2,1)$-cycle in $\widetilde{\Xi}_t$ for $t\in \Delta$, the section $\nu$ satisfies the property (a) and (b).
\end{proof}

\subsection{The main theorem}
Let $r$ be the even number satisfying $2\le r\le 18$.
Let $(\X,\sigma) \rightarrow \hat{U}_r$ be the family of Eisenstein $K3$ surfaces constructed in Section 5.
For each $u \in \hat{U}_r$, we have $\widetilde{\Xi}_u\subset \CH^2(\X_u,1)$ and $\Xi_u\subset J(\X_u,\sigma)$ constructed in Section 4.2.
Then the main theorem is as follows.
\begin{thm}\label{mainthm}
For a very general $u\in U_r$, the image of $\widetilde{\Xi}_u$ generates a rank 2 subgroup in $\CH^2(\X_u,1)_\ind$.
\end{thm}
By Proposition \ref{dominate}, $U_r\rightarrow \M_{r,a}$ is dominant, so we have the following.
\begin{cor}\label{mainintro}
Let $r$ be the even number satisfying $2\le r\le 18$ and we set $a=1+r/2$ if $r\le 10$ and we set $a=11-r/2$ if $r\ge 12$.
For a very general member of $\M_{r,a}$, the corresponding Eisenstein $K3$ surface $(X,\sigma)$ satisfies $\rank \CH^2(X,1)_\ind\ge 2$.
\end{cor}

The key to the proof of Theorem \ref{mainthm} is the following.
\begin{prop}\label{keyprop}
For each $r=2,4,\dots, 18$, we consider the following property.
\begin{equation*}
(\mathbf{N}_r): \begin{aligned}
&\text{There exists an irreducible component $V$ of $\hat{U}_r$}\\
&\text{such that $\Xi_u$ is non-torsion for a very general $u\in V$.}
\end{aligned}
\end{equation*}
Then $(\mathbf{N}_r)$ holds for all $r$.
\end{prop}
First, we see that Theorem \ref{mainthm} follows from Proposition \ref{keyprop}.
\begin{proof}[Proof of Proposition \ref{keyprop} $\Rightarrow$ Theorem \ref{mainthm}]
By Proposition \ref{keyprop}, on some irreducible component $V$ of $\hat{U}_r$, $\Xi_u(\subset J(\X_u,\sigma))$ is non-torsion for a very general $u\in V$.
Since $J(\X_u,\sigma)$ is a $\Z[\zeta]$-module structure by the $\sigma^*$-action and $\Z\Xi_u$ is its non-torsion submodule, the rank of $\Z\Xi_u$ as abelian group is $2$.

Let $V^\circ$ is the smooth locus of $V$.
Then by Proposition \ref{dominate}, the family of Eisenstein $K3$ surfaces $\X|_{V^\circ}\rightarrow V^\circ$ satisfies the assumption in Proposition \ref{nonzeroprop}.
Hence for a very general $u \in V^\circ$, $\nu_E$ factors $\CH^2(\X_u,1)_\ind$.
Since $\nu_E(\widetilde{\Xi}_u) = \Xi_u$ and $\Xi_u$ generates a rank 2 subgroup in $J(\X_u,\sigma)$, $\widetilde{\Xi}_u$ generates a rank 2 subgroup in $\CH^2(\X_u,1)_\ind$ for such $u$.
\end{proof}

The Proposition \ref{keyprop} is proved by the induction.
The proof of the starting case $r=18$ will be postponed to the next section, so we only prove the induction process $(\mathbf{N}_r)\Rightarrow (\mathbf{N}_{r-2})$ here.
\begin{proof}[Proof of $(\mathbf{N}_r)\Rightarrow (\mathbf{N}_{r-2})$]
Let $r$ be either $4,6,\dots, 18$.
By the induction hypothesis, we can take an irreducible components $V$ of $\hat{U}_{r}$ such that $\Xi_u$ is non-torsion for a very general $u\in V$.
We fix a point $u\in V$ such that $\Xi_u$ is non-torsion.
By Proposition \ref{closure}, $u$ is contained in the closure of $\hat{U}_{r-2}$, so we can take a holomorphic map $\rho:\Delta \rightarrow |\O(3,3)|\times |\O(1,0)|\times |\O(0,1)|$ such that $\rho(0) = u$ and $\rho(\Delta^*)\subset \hat{U}_{r-2}$.
Then by Lemma \ref{degenerationlemma}, $\Xi_{\rho(t)}$ is non-torsion for a very general $t\in \Delta^*$.
Hence we have a point $v\in \hat{U}_{r-2}$ such that $\Xi_{v}$ is non-torsion.
Let $V'$ be an irreducible component of $\hat{U}_{r-2}$ which $v$ belongs to.
By Proposition \ref{multisection}, $\Xi = \{\Xi_{v'}\}_{v'\in V'} \subset \J(\X|_{V'},\sigma)$ contains a non-torsion holomorphic local section.
In particular, $\Xi_{v'}$ is non-torsion for a very general $v'\in V'$.
Thus $V'$ satisfies the property $(\mathbf{N}_{r-2})$.
\end{proof}

\section{Case $r=18$}
In this section, we prove the starting case $(\mathbf N_{18})$ in Proposition \ref{keyprop} by calculating the regulator map explicitly.
In the case $r=18$, our higher Chow cycle is essentially constructed in the paper \cite{MT}, although they do not use cycle-theoretic terms.

\subsection{The transcendental regulator map}
Let $X$ be a complex $K3$ surface.
Recall that $\nu: \CH^2(X,1)\rightarrow J(X)$ denotes the regulator map.
To prove the non-torsionness of $(2,1)$-cycles, we need not compute the whole image of the regulator map.
We use the following variant $\Phi$ of the regulator map called the \textit{transcendental regulator map}.
\begin{equation}\label{transreg}
\begin{tikzcd}
\Phi:&[-30pt] \CH^2(X,1)\arrow[r,"\nu"] & J(X) \simeq \dfrac{(F^1H^2(X,\C))^\vee}{H_2(X,\Z)} \arrow[r,twoheadrightarrow] &\dfrac{(H^{2,0}(X))^\vee}{H_2(X,\Z)}
\end{tikzcd}
\end{equation}
where the last projection is induced by $H^{2,0}(X) \hookrightarrow F^1H^2(X,\C)$.

Let $\omega$ be a non-zero holomorphic 2-form on $X$.
Considering the paring with $[\omega]\in H^{2,0}(X)$, we have an isomorphism between the target of (\ref{transreg}) and the abelian group $\C/\Pc(\omega)$, where $\Pc(\omega)$ is the set of periods of $X$ with respect to $\omega$, i.e.,
\begin{equation}\label{perdef}
\Pc(\omega) = \left\{\displaystyle \int_{\Gamma}\omega  \in \C : \Gamma \text{ is a toplogical 2-cycle on }X.\right\}.
\end{equation}
By the formula (\ref{Levineformula}), the transcendental regulator map is calculated as
\begin{equation}\label{transregval}
\Phi(\xi)([\omega]) \equiv  \int_{\Gamma}\omega  \mod \Pc(\omega)
\end{equation}
where $\Gamma$ is a 2-chain associated with $\xi$.

Let $\X\rightarrow S$ be a family of $K3$ surfaces and $\omega$ be a nowhere vanishing relative 2-form.
For an open subset $U$ of $S$ in the classical topology, a period function with respect to $\omega$ on $U$ is a holomorphic function on $U$ given by $U\ni s\mapsto \int_{\Gamma_s}\omega_s$, where $\{\Gamma_s\}_{s\in U}$ is a $C^\infty$-family of 2-cycles on $\X_s$.
The period functions with respect to $\omega$ generate (as a sheaf of abelian groups) the subsheaf $\Pc(\omega)$ of the sheaf $\O_S^\an$ of holomorphic functions. 
Let $\Qc(\omega)$ be the quotient $\O_{S}^\an/\Pc(\omega)$.
For each $s\in S$, the evaluation map $\O_{S}^\an \rightarrow \C; f\mapsto f(s)$ induces the map
\begin{equation*}
\ev_s: \Gamma(S,\Qc(\omega)) \lra \C/\Pc(\omega_s).
\end{equation*}

For an analytic family $\xi = \{\xi_s\}_{s\in S}$ of $(2,1)$-cycles, the right-hand side of (\ref{transregval}) varies holomorphically on $S$.
Thus we have the section $\nu_{\tr}(\xi)\in \Gamma(S,\Qc(\omega))$ such that for each $s\in S$,
\begin{equation*}
\ev_s(\nu_{\tr}(\xi)) = \Phi(\xi_s)([\omega_s]) \quad\text{in } \C/\Pc(\omega_s).
\end{equation*}
We use the following elementary property of sections of $\Qc(\omega)$.
For the proof, see, e.g., \cite{Sat24}, Lemma 2.4.
\begin{lem}\label{importantlemma}
If $\varphi\in \Gamma(S,\Qc(\omega))$ is non-zero, $\ev_s(\varphi)$ is non-zero for a very general $s\in S$.
In particular, if $\nu_{\tr}(\xi)\neq 0$, $\Phi(\xi_s)\neq 0$ for a very general $s\in S$.
\end{lem}

Let $(X,\sigma)$ be an Eisenstein $K3$ surface.
Then we have the inclusion $H^{2,0}(X)\hookrightarrow E(X,\sigma)_\sigma$.
Then the surjection $J(X)\twoheadrightarrow H^{2,0}(X)^\vee/H_2(X,\Z)$ in (\ref{transreg}) factors $J(X,\sigma)$.
This implies the following.
\begin{prop}\label{reduce1}
For a $(2,1)$-cycle $\xi$, if $\Phi(\xi)$ is non-torsion, then $\nu_{E}(\xi)$ is so.
\end{prop}

\subsection{A normalized family}
In this section, we will define a 1-dimensional subfamily $T$ of $\hat{U}_{18}$, such that the natural map $T\times (\mathrm{PGL})^2 \rightarrow \hat{U}_{18}$ is dominant.

Let $T=\P^1-\{0,1,\infty\} = \Spec \C[\lambda,1/\lambda(1-\lambda)]$.
Let $x_0,x_1$ (resp. $y_0,y_1$) be the 1st (resp. 2nd) homogeneous coordinates of $Y=\P^1\times \P^1$ and $x=x_1/x_0$ (resp. $y=y_1/y_0$) be the 1st (resp. 2nd) inhomogeneous coordinates.
Let $\Y = Y\times T$.
Then we define $s \in \Gamma(\Y,\O_\Y(3,3))$ by
\begin{equation}
s=y_0y_1(x_0y_1-x_1y_0)(x_0-x_1)(x_0-\lambda x_1).
\end{equation}
Then for each $\lambda$, $B_\lambda = \div (s_\lambda) \in U_{18}$.
(See Figure. \ref{blambda}.)
Let $F_1$ be the $(1,0)$-divisor $x=0$ and $F_2$ be the $(0,1)$-divisor $y=1$.
Then the map $T\rightarrow \hat{U}_{18}; \lambda \mapsto (B_\lambda, F_1,F_2)$ is well-defined.
By the construction, we see that $(\mathrm{PGL}_2)\times T \rightarrow \hat{U}_{18}$ is surjective.

\begin{figure}
\includegraphics[width = 70mm]{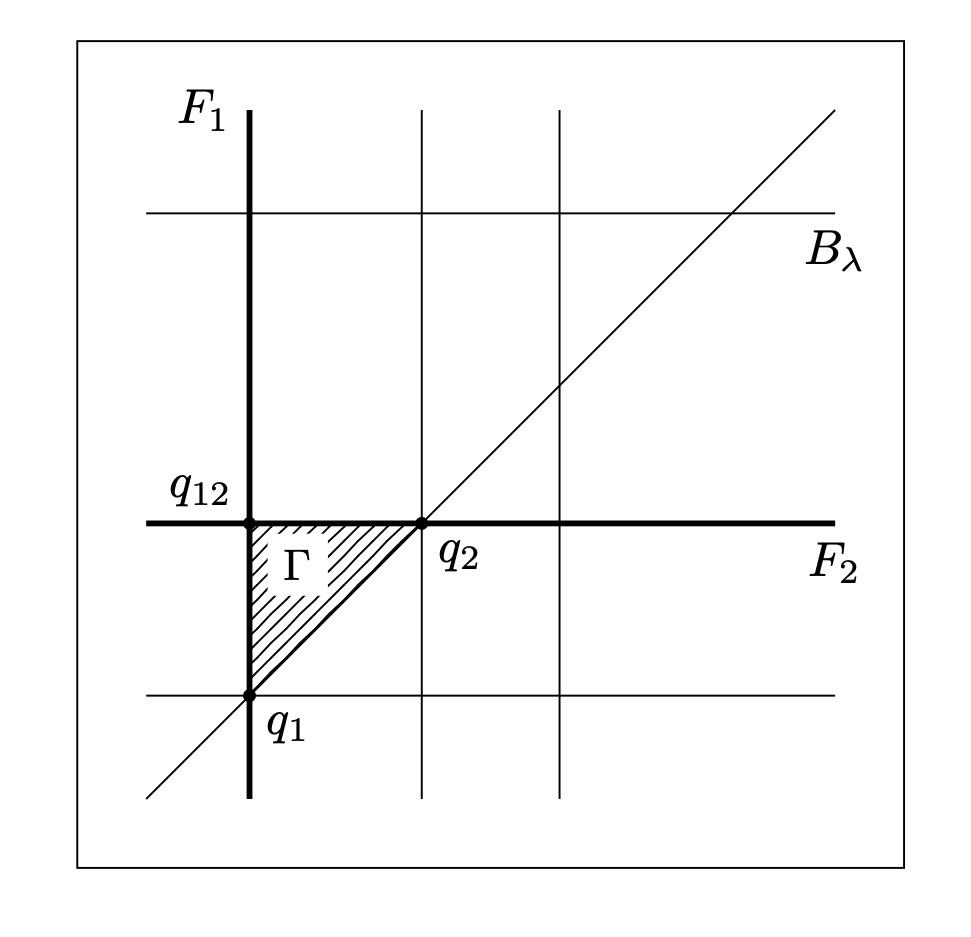}
\caption{The branching locus $B_\lambda$ and curves $F_1,F_2$}
\label{blambda}
\end{figure}

Let $\overline{\X}$ be the triple covering associated with $(\O_\Y(-3,-3),s)$ and $\X\rightarrow \overline{\X}$ be the minimal resolution of singularities.
Then we have a family of Eisenstein $K3$ surfaces $(\X,\sigma)\rightarrow T$ which is the pull-back of the universal family by $T\rightarrow \hat{U}_{18}$.
Furthermore, for each $\lambda\in T$, we have a collection $\widetilde{\Xi}_\lambda$ of $(2,1)$-cycles on $\X_\lambda$.

\begin{equation}
\begin{tikzcd}
\X \arrow[r,"\pi_2","\text{blow-up}"'] \arrow[rr,bend left = 20pt, "\pi"]&[30pt] \overline{\X} \arrow[r,"\pi_1","\text{triple covering}"']  &[30pt] \Y = Y\times T
\end{tikzcd}
\end{equation}
Then we see that on some affine open subset, $\overline{\X}_\lambda$ can be written as
\begin{equation}\label{affineequation}
z^3 = y(y-x)(1-x)(1-\lambda x)
\end{equation}
and the covering transformation $\sigma: \overline{\X}\rightarrow \overline{\X}$ is given by $(x,y,z)\mapsto (x,y,\zeta z)$.
From this equation, we see that $\pi_1^{-1}(q_{12})=\pi_1^{-1}(F_1\cap F_2)$ consists of $(x,y,z)=(0,1,1),(0,1,\zeta),(0,1,\zeta^2)$.

Hence if we set $p_0=(0,1,1)$, we can construct an analytic family of $(2,1)$-cycles $\xi = \{\xi_{\lambda}\}_{\lambda\in T}$ such that $\widetilde{\Xi}_\lambda$ coincides with the set $\{\xi,\sigma_*(\xi),\sigma^2_*(\xi)\}$ up to decomposable cycles.

\subsection{A rational map}
In this section, we define a rational map from a direct product of curves to the $K3$ surface $\X_\lambda$ to relate the periods of $K3$ surfaces to periods of curves.

For each $\lambda \in T$, we define the affine curves $C^\circ_\lambda$ and $E$ by the equations
\begin{equation}
\begin{aligned}
C^\circ_\lambda :&\quad s^3 = t^2(1-t)(1-\lambda t) \\
E:&\quad u^3=v(v-1)
\end{aligned}
\end{equation}
Let $C_\lambda, E$ be the smooth curves which are birational to $C^\circ_\lambda, E$, respectively.
By the equation (\ref{affineequation}) of an affine open subset $\X_\lambda$, we have the following morphism $\varphi$.
\begin{equation}\label{phi}
\begin{tikzcd}
\varphi:C^\circ_\lambda \times E^\circ \arrow[r] & \overline{\X}_\lambda; \quad (s,t,u,v) \arrow[r,mapsto] & (x,y,z) =  (t,vt,us)
\end{tikzcd}
\end{equation}
Then $\varphi$ induces a rational map $C_\lambda\times E\dashrightarrow \X_\lambda$.
We denote this rational map by the same symbol $\varphi$.

We have a natural $\mu_3$-action on $C^\circ_\lambda\times E^\circ$ given by 
\begin{equation}
\begin{tikzcd}
\mu_3\times C^\circ_\lambda\times E^\circ \arrow[r] & C_\lambda^\circ \times E^\circ; \quad (\zeta,s,t,u,v)\arrow[r,mapsto] & (\zeta s,t,\zeta^{-1}u,v).
\end{tikzcd}
\end{equation}
This $\mu_3$-action lifts to $C_\lambda\times E$.
Then $\varphi$ induces the birational map between $(C_\lambda\times E)/\mu_3$ and $\X_\lambda$.

Let $\overline{Z}_\lambda \rightarrow (C_\lambda\times E)/\mu_3$ be the resolution of singularities.
Since $\X_\lambda$ is a $K3$ surface, hence it is minimal, so the birational map $\psi: \overline{Z}_{\lambda}\dashrightarrow \X_\lambda$ induced by $\varphi$ is a birational \textit{morphism}, so it is a composition of blowing-ups.
Let $Z_\lambda$ is the fiber product of $\overline{Z}_\lambda\rightarrow (C_\lambda\times E)/\mu_3$ and the quotient map $C_\lambda\times E \rightarrow (C_\lambda\times E)/\mu_3$.
We denote the natural morphisms $Z_\lambda \rightarrow C_\lambda\times E$ and $Z_\lambda \rightarrow \overline{Z}_\lambda$ by $b$ and $p$, respectively.
Since the quotient map $C_\lambda\times E\rightarrow (C_\lambda\times E)/\mu_3$ is flat, $b:Z_\lambda \rightarrow C_\lambda\times E$ is also a blowing-up.
Furthermore, we have the $\mu_3$-action on $C_\lambda\times E$ lifts to $Z_\lambda$ and $p: Z_\lambda \rightarrow \overline{Z}_\lambda$ coincides with the quotient map of this $\mu_3$-action.

Thus we have constructed following varieties, morphisms and a birational map.
\begin{equation}\label{morphisms}
\begin{tikzcd}
&[-16pt] Z_\lambda \arrow[d,"p","\text{quotient}"'] \arrow[r,"\text{blow-up}"',"b"] &[24pt] C_\lambda \times E \arrow[d,dashrightarrow,"\varphi"] \\
Z_\lambda/\mu_3 \arrow[r,equal]& \overline{Z}_\lambda \arrow[r,"\text{blow-up}"',"\psi"] &  \X_\lambda
\end{tikzcd}
\end{equation}

\subsection{The Picard-Fuchs differential equation}
Let $\omega\in \Gamma(\X, \Omega_{\X/T}^2)$ be the relative 2-form defined by 
\begin{equation}
\omega = \dfrac{dx\wedge dy}{z^2}.
\end{equation}
For $\lambda\in T$, we denote the pull-back of $\omega$ to $\X_\lambda$ by $\omega_\lambda$.
Let $\eta_\lambda=tdt/s^2$ and $\theta=du/v$ be the regular 1-forms on $C_\lambda$ and $E$, respectively.
Under the morphisms in (\ref{morphisms}), we have 
\begin{equation}\label{formrelation}
b^*(pr_1^*(\eta_\lambda)\wedge pr_2^*(\theta)) = (\psi\circ p)^*(\omega_\lambda).
\end{equation}
We define the subgroups $\Pc(\eta_\lambda)$ and $\Pc(\theta)$ of $\C$ by
\begin{equation}
\begin{aligned}
&\Pc(\eta_\lambda) = \left\{\int_\gamma \eta_\lambda : \gamma\text{ is a topological 1-cycle on $C_\lambda$.}\right\},\\
&\Pc(\theta) = \left\{\int_\delta \theta : \gamma\text{ is a topological 1-cycle on $E$.}\right\}
\end{aligned}
\end{equation}
Then we can represent periods of $\X_\lambda$ with respect to $\omega_\lambda$ as follows.
\begin{prop}\label{decomp}
For $\lambda\in T$, let $\Pc(\omega_\lambda)\subset \C$ be the subgroup defined in $(\ref{perdef})$.
Then $\Pc(\omega)$ is contained in the subgroup generated by $\frac{1}{3}c_1c_2\:\: (c_1\in \Pc(\eta_\lambda), c_2\in \Pc(\theta))$.
\end{prop}
\begin{proof}
We have the following morphism of $\Z$-Hodge structures.
\begin{equation*}
\begin{tikzcd}
\phi: H^2(C_\lambda \times E)\arrow[r,"b^*"] &H^2(Z_\lambda) \arrow[r,"(\psi\circ p)_*"] &[10pt]H^2(\X_\lambda)
\end{tikzcd}
\end{equation*}
Since $\psi\circ p$ is generically $3:1$, so $(\psi\circ p)_*\circ (\psi\circ p)^*$ coincides with multiplication by $3$ (cf. \cite{Voi02}, Remark 7.29).
Hence we have $3[\omega_\lambda] = \phi([pr_1^*(\eta_\lambda)\wedge pr_2^*(\theta)])$ by the relation (\ref{formrelation}).
On the other hand, we have the following map between singular homologies.
\begin{equation}\label{singhom}
\begin{tikzcd}
H_2(\X_\lambda,\Z) \arrow[r,"\phi^\vee"] & H_2(C_\lambda \times E,\Z) \arrow[r,twoheadrightarrow] & H_1(C_\lambda,\Z)\otimes_{\Z} H_1(E,\Z),
\end{tikzcd}
\end{equation}
where the first map is the dual of $\phi$ and the second map is the projection induced by the K\"unneth formula.
For a 2-cycle $\Gamma$ on $\X_\lambda$, we denote the image of $[\Gamma]\in H_2(\X_\lambda,\Z)$ under the map (\ref{singhom}) by $\sum_{i,j}c_{i,j}[\gamma_i]\otimes [\delta_j]$, where $\gamma_i$ and $\delta_j$ are 1-cycles on $C_\lambda$ and $E$, respectively.
Then we have 
\begin{equation*}
\int_{\Gamma}\omega_\lambda = \langle[\Gamma],[\omega_\lambda] \rangle = \frac{1}{3}\langle[\Gamma],\phi([pr_1^*(\eta_\lambda)\wedge pr_2^*(\theta)])\rangle = \dfrac{1}{3}\sum_{i,j}c_{i,j}\left(\int_{\gamma_i}\eta_\lambda\right)\left(\int_{\delta_j}\theta\right).
\end{equation*}
Since $\displaystyle\int_{\gamma_j}\in \Pc(\eta_\lambda)$ and $\displaystyle\int_{\delta_j}\theta\in \Pc(\theta)$, we have the result.
\end{proof}

Then we can find the Picard-Fuchs differential operator with respect to $\omega$ (i.e. the differential operator which annihilates the period functions with respect to $\omega$) as follows.

\begin{prop}\label{PFop}
Let $\Dc:\O_T^\an \rightarrow \O_T^\an$ be the differential operator defined by
\begin{equation*}
\Dc_\lambda = \lambda(1-\lambda)\dfrac{d^2}{d\lambda^2} + \left(1-\frac{7}{3}\lambda\right)\dfrac{d}{d\lambda} -\frac{4}{9}.
\end{equation*}
Then for any local section $p$ of $\Pc(\omega)$, we have $\Dc_\lambda(p)=0$.
\end{prop}
\begin{proof}
Let $\Pc(\eta)$ be the subsheaf of $\O_T^\an$ (as a sheaf abelian groups) generated by period functions with respect to $\eta$.
Note that the 1-form $\eta_{\lambda}$ can be identified as 
\begin{equation}
\eta_{\lambda} = \frac{dt}{t^{2/3}(1-t)^{2/3}(1-\lambda t)^{2/3}}
\end{equation}
as a multivalued 1-form.
Let $H(t,\lambda)$ be the multivalued holomorphic function $H(t,\lambda) = -\frac{2}{3}t^{2/3}(1-t)^{1/3}(1-\lambda t)^{-5/2}$. 
Then by the direct calculation, we can check that 
\begin{equation}\label{potential}
\dfrac{\del}{\del t}H(t,\lambda) = \left(\lambda(1-\lambda)\dfrac{\del^2}{\del \lambda^2} +\left(1-\frac{7}{3}\lambda\right)\dfrac{\del}{\del \lambda} -\frac{9}{4}\right)\dfrac{1}{t^{1/3}(1-t)^{2/3}(1-\lambda t)^{2/3}}.
\end{equation}
This implies that for any 1-cycle $\gamma$ on $C_\lambda$, we have
\begin{equation}\label{ann}
0 = \lambda(1-\lambda)\dfrac{\del^2}{\del \lambda^2}\int_{\gamma}\eta_{\lambda} + \left(1-\frac{7}{3}\lambda\right)\dfrac{\del}{\del \lambda}\int_{\gamma}\eta_{\lambda} -\frac{9}{4}\int_{\gamma}\eta_\lambda.
\end{equation}
This implies that the differential operator $\Dc_\lambda$ annihilates any sections of $\Pc(\eta)$.
For a sufficiently small open ball $V$ in $T$, we can take a basis $p_1,p_2,\dots, p_n\in \Pc(\eta)(V)$ of $\Pc(\eta)$.
Then by Proposition \ref{decomp} and the identity theorem, for any section $p\in \Pc(\omega)(V)$, there exists $c_1,c_2,\dots, c_n\in \Pc(\theta)$ such that $p=\frac{1}{3}\sum_{i=1}^n c_ip_i$.
Then by (\ref{ann}), $p$ satisfies $\Dc_\lambda(p) = 0$. 
Hence we get the result.
\end{proof}

Since $\Dc_\lambda$ annihilates all sections of $\Pc(\omega)$, the differential operator $\Dc_\lambda$ factors the sheaf $\Qc(\omega)$.
Thus we have the following morphism of sheaves of abelian groups.
\begin{equation}\label{DCQ}
\Dc_\lambda: \Qc(\omega) \lra \O_T^\an
\end{equation}

\subsection{Calculation of the regulator map}
In this section, we fix $\lambda \in\R$ such that $0<\lambda <1$ and compute the image of the $(2,1)$-cycle $\xi_\lambda$ under the transcendental regulator map.

Let $B_1$ be the subvariety $x=y$ on $Y$ and for $i=1,2$, set $q_i=B_1\cap F_i$ and $q_{12}=F_1\cap F_2$.
Let $D_1$ be the unique (integral) subvariety on $\X_\lambda$ such that $\pi(D_1)=B_1$
For $i=1,2$, $C_i$ denotes the strict transformation of $\pi_1^{-1}(F_i)$ and $E_{q_i}$ denotes the component of exceptional divisors over $q_i$ such that $E_{q_i}\cap D_1\neq  \emptyset$.
Then $E_{q_i}$ and $F_i$ (resp. $E_{q_i}$ and $B_1$) intersect at one point, so we denote this point $p_{i,0}$ (resp. $p_{i,1}$.
Note that we have $p_{i,0}\neq p_{i,1}$.
Let $p_0\in \X_\lambda$ be the point which is the unique lift of $(x,y,z)=(0,1,1)\in \overline{\X}_\lambda$.
For $i=1,2$, we set $p_i = \sigma^i_*(p_0)$.
Then we have $C_1\cap C_2 = \{p_0,p_1,p_2\}$.

Let $\Gamma = \{(x,y)\in \R^2\: |\: 0\le x\le y\le 1\}$ and $\Gamma^\circ$ be its interior.
We regard $\Gamma$ as a 2-chain on $\P^1\times \P^1$.
(See Figure. \ref{blambda}.)
Then we have the following.
\begin{prop}\label{2chainprop}
The image of $\xi_{\lambda} \in\CH^2(\X_\lambda,1)$ under the transcendental regulator map is given as follows:
\begin{equation}
\Phi(\xi_{\lambda})([\omega_\lambda]) \equiv (1-\zeta)\int_{\Gamma^\circ} \dfrac{dx\wedge dy}{\left\{y(y-x)(1-x)(1-\lambda x)\right\}^{2/3}} \mod \Pc(\omega_\lambda)
\end{equation}
where $\left\{y(y-x)(1-x)(1-\lambda x)\right\}^{2/3}$ denotes the square of the positive cubic root of $y(y-x)(1-x)(1-\lambda x)\in \R_{>0}$.
\end{prop}
\begin{proof}
Recall that $\xi_{\lambda}$ is represented by $(C_1,f_1) + (C_2,F_2)$.
For $i=0,1$, let $\gamma_{i}$ be the pull-back of the path $[\infty,0]$ (see Section 2.2) on $\P^1$ by $f_{i}:C_{i}\rightarrow \P^1$.
By the relation (\ref{divrelation}), $\gamma_0$ (resp. $\gamma_1$) is the path from $p_0$ to $p_1$ (resp. $p_1$ to $p_0$).
We will find a 2-chain associated with $\xi_{\lambda}$, in other words, a 2-chain on $\X_{\lambda}$ whose boundary coincides with $\gamma_0+\gamma_1$.
 
Since $\Gamma^\circ$ does not intersects with branching locus, $\pi^{-1}(\Gamma^\circ)$ decompose into the disjoint union of three copies $\Gamma^\circ_0, \Gamma^\circ_1$ and $\Gamma^\circ_2$ of $\Gamma^\circ$.
We name them so that the point $p_i$ is contained in the closure of $\Gamma^\circ_i$.
We define $\Gamma_i = \overline{\Gamma^\circ_i}$ for $i=0,1,2$.
For $i=1,2,3$, we name the paths appearing in the boundary of $\Gamma_i$ as follows.
(See Figure. \ref{top}.)

\begin{figure}
\includegraphics[width = 95mm]{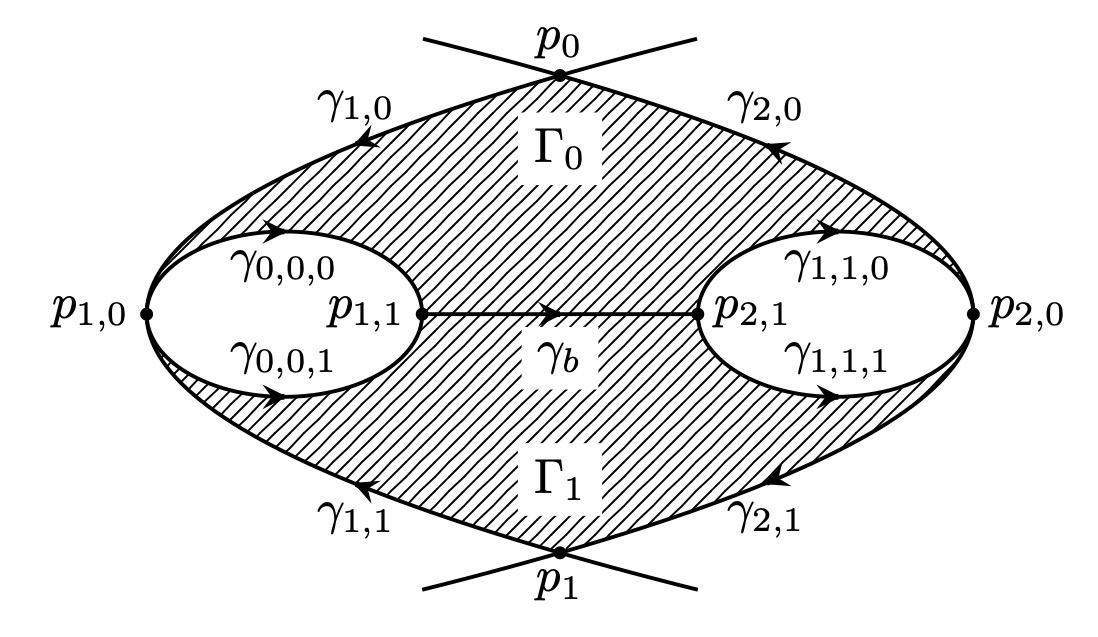}
\caption{The topological 2-chain $\Gamma_0, \Gamma_1$ and paths appearing their boundaries}
\label{top}
\end{figure}
\begin{enumerate}
\item The path $\gamma_{b}$ on $D_1$ is the unique lift of the paths $[0,1]\rightarrow Y; t\mapsto (x,y)=(t,t)$.
This is the path from $p_{1,1}$ to $p_{2,1}$.
\item The path $\gamma_{0,0,i}$ is a path on $E_{q_1}$ from $p_{1,0}$ to $p_{1,1}$.
Using the local coordinates $w=y/z$ on $E_{q_1}$, this can be expressed by $[0,1]\rightarrow E_{q_1}; t\mapsto w=\zeta^{-i} \frac{t}{1-t}$.
\item The path $\gamma_{1,1,i}$ is a path on $E_{q_2}$ from $p_{2,1}$ to $p_{2,0}$. 
Using the local coordinates $w'=z/1-x$ on $E_{q_2}$, this can be expressed by $[0,1]\rightarrow E_{q_2}; t\mapsto w'=\zeta^i \frac{t}{1-t}$.
\item The path $\gamma_{1,i}$ is the path on $C_1$ from $p_i$ to $p_{1,0}$, which is the lift of $[0,1]\rightarrow Y;t\mapsto (x,y)=(0,1-t)$.
\item The path $\gamma_{2,i}$ is the path on $C_2$ from $p_{2,0}$ to $p_{i}$, which is the lift of $[0,1]\rightarrow Y;t\mapsto (x,y)=(1-t,1)$.
\end{enumerate}
These paths satisfies $\del\Gamma_i = \gamma_{1,i} + \gamma_{0,0,i} + \gamma_{b} + \gamma_{1,1,i}+\gamma_{2,i}$ for $i=1,2$.

Since $E_{q_1}$ (resp. $E_{q_2}$) are isomorphic to $\P^1$ and $H_1(\P^1,\Z)=0$, we have a topological 2-chain $\Gamma_{q_1}$ (resp. $\Gamma_{q_2}$) on $E_{q_1}$ (resp. $E_{q_2}$) satisfying $\del \Gamma_{q_1} = \gamma_{0,0,0}-\gamma_{0,0,1}$ (resp. $\del \Gamma_{q_2} = \gamma_{1,1,0}-\gamma_{1,1,1}$).
Similarly, since $C_1$ (resp. $C_2$) are isomorphic to $\P^1$, we have a topological 2-chain $\Gamma_{c,1}$ (resp. $\Gamma_{c,2}$) satisfying $\del\Gamma_{c,1} = \gamma_{1,0}-\gamma_{1,1} - \gamma_{1}$ (resp. $\del\Gamma_{c,2} = \gamma_{2,0}-\gamma_{2,1} - \gamma_{2}$).
Then we have the following equality of 1-cycles on $\X_\lambda$.
\begin{equation}\label{rel2chain}
\del(\Gamma_0-\Gamma_1-\Gamma_{q_1}-\Gamma_{q_2}-\Gamma_{c,1}-\Gamma_{c,2}) = \gamma_{1}+\gamma_{2}
\end{equation}
Thus, by the relation (\ref{rel2chain}), the 2-chain $\Gamma_0-\Gamma_1-\Gamma_{q_1}-\Gamma_{q_2}-\Gamma_{c,1}-\Gamma_{c,2}$ is a 2-chain associated with $\xi_{\lambda}$.
By the formula (\ref{transregval}), we have 
\begin{equation}\label{regval1}
\Phi(\xi_{\lambda})([\omega_\lambda]) \equiv \int_{\Gamma_0}\omega_\lambda - \int_{\Gamma_1}\omega_\lambda -\int_{\Gamma_{q_1}}\omega_\lambda -\int_{\Gamma_{q_2}}\omega_\lambda -\int_{\Gamma_{c,1}}\omega_\lambda -\int_{\Gamma_{c,2}}\omega_\lambda
\end{equation}
Since $\Gamma_{q_1}, \Gamma_{q_2}, \Gamma_{c,1}$ and $\Gamma_{c,2}$ are 2-chains on 1-dimensional subvarieties, the restriction of $\omega_\lambda$ vanishes.
Furthermore, since $\Gamma_1 = \sigma(\Gamma_0)$ and $\sigma^*(\omega_\lambda) = \zeta\omega_{\lambda}$, the righthand side of (\ref{regval1}) can be computed as
\begin{equation}
\int_{\Gamma_0}\omega_{\lambda} - \int_{\Gamma_1}\omega_{\lambda}  =  \int_{\Gamma_0}\omega_{t} -  \int_{\Gamma_0}\zeta\omega_{\lambda} = (1-\zeta)\int_{\Gamma_0}\omega_{\lambda}.
\end{equation}
Since the integration of $\omega_{\lambda}$ on $\Gamma_0-\Gamma^\circ_0$ is 0, we have 
\begin{equation}
 (1-\zeta)\int_{\Gamma_0}\omega_{\lambda} = (1-\zeta)\int_{\Gamma_0^\circ}\omega_\lambda = (1-\zeta)\int_{\Gamma^\circ}\dfrac{dx\wedge dy}{\left\{y(y-x)(1-x)(1-\lambda x)\right\}^{2/3}}.
\end{equation}
Combining the above equations, we get the result.
\end{proof}

For $\lambda\in \R$, such that $0<\lambda < 1$, we define the function $G(\lambda)$ by the improper integral
\begin{equation}\label{Glambda}
G(\lambda) = (1-\zeta)\int_{\Gamma^\circ}\dfrac{dx\wedge dy}{\left\{y(y-x)(1-x)(1-\lambda x)\right\}^{2/3}}.
\end{equation}
Note that the integration converges by Proposition \ref{2chainprop}.
Then we see that the function $G(\lambda)$ satisfies the following.

\begin{prop}\label{diffeqnG}
We can extend $G(\lambda)$ to the multivalued holomorphic function on $T$ satisfying the differential equation
\begin{equation}\label{diffeqnGeqn}
\Dc_\lambda(G(\lambda))=  \frac{1-\zeta}{\lambda-1}
\end{equation}
where $\Dc_\lambda$ is the Picard-Fuchs differential operator in Proposition \ref{PFop}.
\end{prop}
\begin{proof}
It is enough to show that (\ref{diffeqnGeqn}) holds for $0<\lambda <1$.
By the coordinate transformation $y=vx$, the integral region of the right-hand side of (\ref{Glambda}) is transformed in to $\triangle = \{(x,v)\in \R^2\:|\:0<x<1\text{ and }1<v<1/x\}$, so we have
\begin{equation}\label{sekibun}
G(\lambda) = (1-\zeta)\int_{\triangle}  \dfrac{dx\wedge dv}{x^{1/3}(1-x)^{2/3}(1-\lambda x)^{2/3}v^{2/3}(v-1)^{2/3}}.
\end{equation}
We will consider $\Dc_\lambda(G(\lambda))$.
By the Lebesgue's dominated convergence theorem, we can interchange the action of differential operator $\Dc_\lambda$ and the integration, so we have 
\begin{equation}\label{sekibun2}
\Dc_\lambda(G(\lambda)) = \int_{\triangle} \Dc_\lambda\left(\dfrac{dx\wedge dv}{x^{1/3}(1-x)^{2/3}(1-\lambda x)^{2/3}v^{2/3}(v-1)^{2/3}}\right).
\end{equation}
For a sufficiently small positive real number $\varepsilon$, let $\triangle_\varepsilon$ be the open subset of $\triangle$ defined by $\triangle_\varepsilon = \{(x,v)\in \triangle \:|\: \varepsilon < x < 1-\varepsilon\}$.
Since (\ref{sekibun2}) converges in the sense of Lebesgue, we have
\begin{equation}
= \lim_{\varepsilon \rightarrow 0}\left\{(1-\zeta)\int_{\triangle_\varepsilon}  \Dc_\lambda\left(\dfrac{dx\wedge dv}{x^{1/3}(1-x)^{2/3}(1-\lambda x)^{2/3}v^{2/3}(v-1)^{2/3}}\right)\right\}.
\end{equation}
By the relation (\ref{potential}), we have
\begin{equation}
= \lim_{\epsilon \rightarrow 0} \left\{-\frac{2}{3}(1-\zeta)\int_{\triangle_\varepsilon} d\left(\frac{x^{2/3}(1-x)^{1/3}}{(1-\lambda x)^{5/3}}\cdot \frac{dv}{v^{2/3}(v-1)^{2/3}}\right)\right\}.
\end{equation}
By the Stokes theorem, we have 
\begin{equation}
=  \lim_{\epsilon \rightarrow 0} \left\{ -\frac{2}{3}(1-\zeta)\int_{\del\triangle_{\varepsilon}} \frac{x^{2/3}(1-x)^{1/3}}{(1-\lambda x)^{5/3}}\cdot \frac{dv}{v^{2/3}(v-1)^{2/3}}\right\}.
\end{equation}
When $\varepsilon \rightarrow 0$, the integrand vanishes or converges to $0$ on $\del \triangle_\varepsilon$ except on $v=1/x$, hence we have 
\begin{equation}
=-\frac{2}{3}(1-\zeta)\int_{x=0}^{x=1}\frac{dx}{(1-\lambda x)^{5/3}(1-x)^{1/3}} =\frac{1-\zeta}{1-\lambda} \int_{w=1}^{w=0}dw =  \frac{1-\zeta}{\lambda-1}
\end{equation}
where we use the coordinate transformation $w=(1-x)^{2/3}(1-\lambda x)^{2/3}$ in the second last equality.

\end{proof}
Using this proposition, we can show the following.

\begin{thm}\label{tnonzero}
For a very general $\lambda\in T$, $\Phi(\xi_{\lambda})$ is non-torsion.
\end{thm}
\begin{proof}
By Proposition \ref{2chainprop}, we see that 
\begin{equation}\label{coincideeqn}
\ev_{\lambda}(\nu_\tr(\xi)) = \Phi(\xi_{\lambda})([\omega_\lambda]) \equiv G(\lambda) \mod \Pc(\omega_\lambda).
\end{equation}
holds for $0<\lambda <1$.
By Lemma \ref{importantlemma}, $\nu_{\tr}(\xi)$ is represented by the multivalued holomorphic function $G(\lambda)$.
Then we have $\Dc_\lambda(\nu_{\tr}(\xi)) = \Dc_\lambda(G(\lambda))$. 
Note that we regard $\Dc_\lambda$ on the left-hand side as the map $\Gamma(T,\Qc(\omega))\rightarrow \Gamma(T,\O_T^\an)$ induced by the morphism (\ref{DCQ}).
By Proposition \ref{diffeqnG}, $\Dc_\lambda(G(\lambda))\neq 0$, so $\nu_{\tr}(\xi)$ is non-torsion.
Then by Lemma \ref{importantlemma}, we have the result.
\end{proof}

Then we can prove the property $(\mathbf{N}_{18})$ 
\begin{proof}[Proof of $(\mathbf{N}_{18})$]
By Theorem \ref{tnonzero}, for very general $\lambda \in T$, $\Phi(\xi_{\lambda})$ is non-torsion.
Then $\nu_{E}(\xi_{\lambda})$ is non-torsion for such a $\lambda$ by Proposition \ref{reduce1}.
Since $\nu_{E}(\xi_{\lambda})\in \Xi_\lambda$, $\Xi_\lambda \subset J(\X_\lambda,\sigma)$ is also non-torsion.
Thus we have proved the property $(\mathbf{N}_{18})$.
\end{proof}

\bibliographystyle{plain}
\bibliography{reference}

\end{document}